\documentclass[12,reqno]{amsart}
\usepackage{amsfonts}
\usepackage{amssymb}
\usepackage{a4wide}
\usepackage{pgf,pgfarrows,pgfautomata,pgfheaps,pgfnodes,pgfshade}
\usepackage[colorlinks]{hyperref}
\usepackage{lipsum}
\usepackage{url}

\usepackage{ulem}
\usepackage[numbers]{natbib}
\usepackage[pdftex]{}

\usepackage{hyperref}
\numberwithin{equation}{section}

\newtheorem{defi}{Definition}[section]
\newtheorem{thm}[defi]{Theorem}
\newtheorem{lemm}[defi]{Lemma}
\newtheorem{remark}[defi]{Remark}
\newtheorem{cor}[defi]{Corollary}
\newtheorem{assum}[defi]{Assumption}
\newtheorem{prop}[defi]{Proposition}

\newcommand{\dive}{\operatorname{div}}

\newcommand{\X}{\mathcal{X}}

\DeclareMathOperator{\sgn}{sgn}
\DeclareMathOperator{\essup}{essup}

\begin{document}

		\title{Differentiability of quadratic forward-backward SDEs with rough drift}
		
		\author{Peter Imkeller}
		\address{Humboldt-Universit\"{a}t zu Berlin, Institut f\"{u}r Mathematik, Unter den Linden 6, D-10099
			Berlin, Germany}
		\email{imkeller@mathematik.hu-berlin.de}
		\thanks{P. Imkeller was supported in part by DFG Research Unit FOR 2402.}
		\author{Rhoss Likibi Pellat}
		\address{Department of Mathematics, University of Ghana, Accra, Ghana}
		\address{African Institute for Mathematical Sciences, Ghana}
		\email{rhoss@aims.edu.gh}
		\thanks{R. Likibi Pellat is funded by DAAD under the programme PhD-study at AIMS. This work was initiated when R. L. P  visited the Humboldt University of Berlin. He is grateful for the hospitality. He also thanks Prof Stefan Geiss for helpful comments and suggestions}
		\author{Olivier Menoukeu Pamen}
		\address{African Institute for Mathematical Sciences, Ghana}
		\address{Institute for Financial and Actuarial Mathematics, Department of Mathematical Sciences, University of Liverpool, L69 7ZL, United Kingdom}
		\email{menoukeu@liverpool.ac.uk}
		\thanks{O. Menoukeu Pamen acknowledges the funding provided by the Alexander von Humboldt Foundation, under the program financed by the German Federal Ministry of Education and Research entitled German Research Chair No 01DG15010.}
		
		
		\date{\today}
		

		\maketitle

		\section*{Abstract}

		In this paper, we consider quadratic forward-backward SDEs (QFBSDEs), for {which} the drift in the forward equation does not satisfy the standard globally Lipschitz condition and the driver of the backward system {possesses} nonlinearity of type $f(|y|)|z|^2,$ where $f$ is any locally integrable function. We prove both the Malliavin and classical derivative of the QFBSDE and provide representations of these processes. We study a numerical approximation of this system in the sense of \cite{ImkDosReis} in which the authors assume that the drift is Lipschitz and the driver of the BSDE is quadratic in the traditional sense (i.e., $f$ is a positive constant). We show that the rate of convergence is the same as in \cite{ImkDosReis}.

		\section{Introduction}
		In this paper, we address the problem of the Malliavin and the classical differentiability of a class of quadratic forward-backward SDEs (FBSDEs) with rough drift. FBSDEs have attracted a lot of interest in the last four decades due to their applications to optimal control, financial/insurance mathematics and the theory of PDE via the non-linear Feynman-Kac formula. Of particular interest is the class of BSDEs whose drivers grow quadratically in the control variable $Z$. Such a BSDE appears for example in exponential utility maximisation or in the Epstein-Zin utility maximisation problems. To the best of our knowledge, the first result on existence and uniqueness of BSDEs with quadratic drivers and bounded terminal value is due to Kobylanski (\cite{Kobylansky}). This result was extended to the case of unbounded terminal value in \cite{BriandHu, BriandHu08} and in other different ways in \cite{HuImkMuller, Morlais09,DelbaenHuRichou,BarrieuElkaroui}.

		Recently the authors in \cite{Ouknine} studied a new class of unbounded quadratic BSDE {for which} the generator $g$ has the following growth condition $|g(t,y,z)| \leq C(1 + |y| + f(y)|z|^2)$, where $f$ is an integrable function. Their approach is based on exponential-type transformation and an It\^o-Krylov formula for BSDEs. This result was extended in \cite{Bahlali1} to the case of locally integrable function $f$ by using the so called {\it domination method}.

		Another question of importance that arises in the study of BSDEs is the characterisation of the control process $Z$. {If} the coefficients of the FBSDEs are smooth enough, $Z$ is given as the "derivative", either in the sense of the non linear Feynman-Kac formula or in the Malliavin sense via the Clark-Ocone-Hausmann formula. Moreover, $(Z_t)_{t\in [0,T]}$ has a continuous version given by the Malliavin derivative of the backward equation $(D_t Y_t)_{t\in [0,T]}$. The later turns out to be a crucial concept when one deals with novel discretisations of BSDEs implemented with deep learning regressions (see for instance \cite{NAO21}). When the parameters are not smooth, the existence of Malliavin and classical differentiable solutions to FBSDEs remain a challenging question. 
		
		Smoothness of  solutions of FBSDEs, in both the classical and Malliavin sense when the coefficients are smooth enough were established in \cite{Pardoux-Peng92, KPQ97}. In the quadratic case, the first results were derived by the authors in \cite{ArImDR} assuming that the driver is of the form $g(t,x,y,z)= \ell(t,x,y,z) + \alpha|z|^2 $ with $\ell \in C^1$ and Lipschitz continuous in $(x,y,z)$. This work was then extended to non-linear quadratic generators in \cite{ImkDos}. Their method for establishing the classical differentiability relies on differences of difference quotients together with the completeness of the vector space. In order to prove the Malliavin differentiability, they proposed an approximation procedure via a family of truncated BSDEs and a compactness criterion argument. Using a technique based on Kunita's method (solving an abstract BSDE with stochastic Lipschitz conditions), the authors in \cite{BriandConfortola} proved a differentiability result in the classical sense for solutions to quadratic FBSDEs.  Other related results include the work \cite{ImkRevRich} for differentiability of FBSDE driven by continuous martingales with quadratic growth and the work \cite{FreiDosReis} for classical differentiability of FBSDEs with polynomial growth (see also \cite{MPR17}). 
    
    	The results obtained in the aforementioned papers  do not cover the case of non uniformly Lipschitz drift coefficients. 
In \cite{RhOlivOuk}, the authors established a result on Malliavin derivative solutions to coupled FBSDEs with discontinuous coefficients. Their method exploits the regularization effect of the Brownian motion when the diffusion coefficient is a constant and the regularity of the {\it weak decoupling field } combined with a compactness criterion argument. However, the authors did not provide a representation satisfied by the Malliavin derivatives in terms of BSDEs. This is mainly due to the very mild assumptions that were considered there.



Another motivation of this paper concerns the numerical approximation of the solution to quadratic FBSDEs with rough drift. More precisely, we consider the following Markovian-type BSDE 
\begin{align}\label{illu}
\begin{cases}
	X_t = x + \displaystyle \int_{0}^{t} b(s,X_s)\mathrm{d}s + \int_{0}^{t} \sigma(s,X_s)\mathrm{d} B_s,\\
	Y_t =  \phi(X_T) + \displaystyle \int_{t}^{T} g(s,X_s,Y_s,Z_s)\mathrm{d}s - \int_{t}^{T} Z_s\mathrm{d}B_s.
\end{cases}
\end{align}
The explicit solutions to equation \eqref{illu} are generally unknown, thus, developing numerical methods to approximate these solutions becomes a good  alternative. Assuming that the coefficients in \eqref{illu} are Lipschitz continuous, such problem was studied in \cite{Chev97} (see also \cite{BouTou04, Zhang041}). To the best of our knowledge, the case with quadratic drivers was first investigated in \cite{ImkDos}. Their method consists in approximating solution of \eqref{illu} by the solution $(Y^n,Z^n)$ of a family of truncated BSDEs given by:
\begin{align}\label{illu1}
Y_t^n =  \phi(X_T) + \displaystyle \int_{t}^{T} g(s,X_s,Y_s^n,\rho_n(Z_s^n))\mathrm{d}s - \int_{t}^{T} Z_s^n\mathrm{d}B_s,\quad n \in \mathbb{N},
\end{align}
where $\rho_n$ is a smooth real valued function given by \eqref{trunc}. The truncation errors is studied in the natural norms of the solutions to \eqref{illu} and \eqref{illu1}. Its proof uses the facts $Z^n_t$ has a continuous version  given by  $D_tY_t^n = (\nabla_x Y_t^n)(\nabla_x X_t)^{-1}$ such that $\sup_{n \in \mathbb{N}} \| Z^n\|_{\mathcal{S}^p} < \infty$.
 
By assuming that the drift is in addition H\"older continuous in $t$ and the diffusion is a differentiable function of time, the author in \cite{Richou} improved the results obtained in \cite{ImkDos} by considering a different scheme. More precisely, the author approximates BSDEs with quadratic generators by BSDEs with Lipschitz generators. The above result was generalised in \cite{ChasRich} for time homogeneous drift and diffusion coefficient satisfying a global Lipchitz condition. 

In this paper, we assume that the driver $g$ is stochastic Lipschitz in $y$ and quadratic of the form $f(|y|)|z|^2$, where $f$ is any increasing and locally integrable function and the drift is either measurable and bounded or bounded and H\"{o}lder continuous. We then study both differentiablity in the Malliavin and classical sense of the solution to the quadratic FBSDE. We follow the method developed in \cite{ArImDR, ImkDos} and work under much weaker conditions. We also provide the representations of both the Malliavin and classical derivatives of the solution to the FBSDE. One of the main obstacles in this work is the non smoothness of the drift. 
In the case of $\beta$-H\"older continuous drift coefficients, using the fact that the first variation process is $\beta$-H\"older continuous in $x \in \mathbb{R}^d$ (see \cite{FlanGubiPrio10}), we derive the representation of classical derivative of the BSDE.

In this work we also provide a truncation error of the numerical approximation. Since the driver $g$ is only stochastic Lipschitz in $Y$ we also have to truncate the driver in $Y$. Another difficulty is to find the bound of the supremum norm of $Z^n$, which depends on the supremum norm of the inverse of the first variation process of the forward equation. We tackle this problem by using once more the regularity of the solution to the Kolmogorov equation associated to the forward equation. This allows us to obtain a transformed SDE with bounded and smooth coefficients from which the sought bound can be obtained. We then achieve the same rate of convergence as in \cite{DosReis} (see Theorem \ref{thmconve1}). As an independent result, we also obtain a Zhang's path regularity theorem in our general framework. 



The remaining part of the paper is organised as follows: in Section \ref{Genset}, we provide basic definitions and results on BSDEs. The existence and uniqueness results as well as the main a priori estimates are given in Section \ref{sectexun}. In Section \ref{sectdiff} we study the differentiability for an abstract BSDEs with stochastic Lipschitz generators Section \ref{Appli} is devoted to the smothness of the solution of the FBSDE whereas Section \ref{NA} is concerned with rate of convergence of the numerical approximation of the solution to the BSDE. 
\section{General settings and Notations}\label{Genset}
\subsection{Some notation}
Throughout this paper a stochastic basis $(\Omega,\mathfrak{F},\{\mathfrak{F}_t\}_{t\geqslant 0},\mathbb{P}, \{B_t\}_{t\geq 0})$ is given. Here $\{\mathfrak{F}_t\}_{t\geqslant 0}$ is the standard filtration generated by the $d\text{-dimensional}$ Brownian motion $\{B_t\}_{t\geq 0}$ augmented by all $\mathbb{P}\text{-null}$ sets of $\mathfrak{F}.$ 
For fixed $T> 0,d \in \mathbb{N},$ $p\in [2,\infty)$, we denote by:
\begin{itemize}
\item $L^p(\mathbb{R}^d)$ the space of $\mathfrak{F}_T\text{-adapted}$ random variables  $X$ such that $\Vert X \Vert^p_{\tiny L^p }:= \mathbb{E}|X|^p < \infty$; 
\item $L^{\infty}(\mathbb{R}^d)$ the space of bounded random variables with norm $\Vert X \Vert_{\tiny L^{\infty} }:= \essup_{\omega\in \Omega}|X(\omega)|$; 
\item $\mathcal{S}^p(\mathbb{R}^d)$ the space of all adapted continuous $\mathbb{R}^d$-valued processes $X$ such that $\Vert X \Vert^p_{\tiny \mathcal{S}^p(\mathbb{R}^d) }:= \mathbb{E}\sup_{t\in [0,T]} |X_t|^p < \infty$; 
\item $\mathcal{H}^p(\mathbb{R}^d)$ the space of all predictable $\mathbb{R}^d$-valued processes $Z$  such that $\Vert Z \Vert^p_{\tiny \mathcal{H}^p(\mathbb{R}^d) }:= \mathbb{E}(\int_0^T |Z_s|^2\mathrm{d}s)^{p/2}  < \infty$; 
\item $\mathcal{S}^{\infty}(\mathbb{R}^d)$ the space of continuous $\{ \mathfrak{F}_s \}_{0\leq t\leq T} $-adapted processes $Y:\Omega\times[0,T]\rightarrow \mathbb{R}^d$ such that $\Vert Y \Vert_{\tiny \infty }:= \essup_{\omega\in \Omega} \sup_{t\in [0,T]} |Y_t(\omega)| < \infty;$
\item BMO($\mathbb{P}$) the space of square integrable martingales $M$ with $M_0 = 0$ such that $ \Vert M\Vert_{\text{BMO}(\mathbb{P})} = \sup_{\tau \in [0,T]} \Vert \mathbb{E} [\langle M \rangle_T  - \langle M \rangle_{\tau}]/\mathfrak{F}_{\tau} \Vert_{\infty}^{1/2} < \infty,$ the supremum is taken over all stopping times $\tau \in [0,T];$
\item $\mathcal{H}_{\text{BMO}}$ the space of $\mathbb{R}^d\text{- valued}$ $\mathcal{H}^p\text{-integrable} $ processes $(Z_t)_{t\in [0,T]}$ for all $p\geq 2$ such that $Z*B = \int_0 Z_s\mathrm{d}B_s \in \text{BMO}(\mathbb{P}).$ We define $\Vert Z \Vert_{\mathcal{H}_{\text{BMO}}}:= \Vert \int Z\mathrm{d}B \Vert_{\text{BMO}(\mathbb{P})}$;
\item $L^{\infty}([0,T];C_b^{\beta}(\mathbb{R}^d;\mathbb{R}^d))$ the space of all vector fields $b:[0,T]\times \mathbb{R}^d\rightarrow \mathbb{R}^d$ having all components in $L^{\infty}([0,T];C_b^{\beta}(\mathbb{R}^d))$ and $L^{\infty}([0,T];C_b^{\beta}(\mathbb{R}^d))$ stands for the set of all bounded Borel functions $b:[0,T]\times \mathbb{R}^d\rightarrow \mathbb{R}$ such that 
\[ [b]_{\beta,T} = \sup_{t \in [0,T]} \sup_{x\neq y\in \mathbb{R}^d} \frac{|b(t,x) -b(t,y)|}{|x-y|^\beta} < \infty . \]
\end{itemize}

Below, we briefly introduce the spaces of Malliavin differentiable random variables $\mathbb{D}^{k,p}.$  For more information on Malliavin calculus we refer the reader to \cite{DOP08, Nua06}.
Let $\mathcal{S}$ be the space of random variables $\xi$ of the form 
\[ \xi = F\Big( (\int_{0}^{T} h_s^{1,i}\mathrm{d}W_s^1)_{1\leq i\leq n},\cdots, (\int_{0}^{T} h_s^{d,i}\mathrm{d}W_s^d)_{1\leq i\leq n}   \Big),  \]
where $F\in C_b^{\infty}(\mathbb{R}^{n\times d}), h^1,\ldots,h^n \in L^2([0,T];\mathbb{R}^d)$ and $n\in \mathbb{N}.$ For simplicity, we assume that all $h^j$ are written as row vectors. For $\xi \in \mathcal{S},$ we define $D=(D^1,\cdots,D^d):\mathcal{S}\rightarrow L^2(\Omega\times [0,T])^d$ by 
\[ 
D^i_{\theta}\xi := \sum_{j=1}^{n}\frac{\partial F}{\partial x_{i,j}} \Big( \int_{0}^{T} h_t^{1}\mathrm{d}W_t,\cdots,\int_{0}^{T} h_t^{n}\mathrm{d}W_t   \Big)h_{\theta}^{i,j}, 0\leq \theta \leq T, 1\leq i \leq d, 
\]
and for $k\in \mathbb{N}$ and $\theta=(\theta_1,\cdots,\theta_k)\in [0,T]^k$ its $k\text{-fold}$ iteration as 
\[ D^{(k)} = D^{i_1}\cdots  D^{i_k}_{1\leq i_1,\cdots,i_k\leq d}.  \]
For $k\in \mathbb{N}$ and $p\geq 1,$ let $\mathbb{D}^{k,p}$ be the closure of $\mathcal{S}$ with respect to the norm 
\[ \| \xi \|^p_{k,p} = \| \xi \|^p_{L^p} + \sum_{i=1}^{k} \| |D^{(i)}\xi|\|^p_{(\mathcal{H}^p)^i}  .\]
The operator $D^{(k)}$ is a closed linear operator on the space $\mathbb{D}^{k,p}.$ Observe that if $\xi \in \mathbb{D}^{1,2} $ is $\mathfrak{F}_t\text{-measurable}$ then $D_{\theta}\xi = 0$ for $\theta\in (t,T].$ Further denote $\mathbb{D}^{k,\infty} =\cap_{p>1}\mathbb{D}^{k,p}.$

For $k\in \mathbb{N}, p\geq 1,$ denote by $\mathbb{L}_{k,p}(\mathbb{R}^m)$ the set of $\mathbb{R}^m\text{-valued}$ progressively measurable processes $u=(u^1,\cdots,u^m)$ on $[0,T]\times \Omega$ such that 
\begin{itemize}
\item[(i)] For Lebesgue a.a. $t \in [0,T], u(t,\cdot)\in (\mathbb{D}^{k,p})^m; $
\item[(ii)] $(t,\omega)\rightarrow D_{\theta}^k u(t,\omega) \in (L^2([0,T]^{1+k}))^{d\times m}$ admits a progressively measurable version;
\item[(iii)] $\|u \|_{k,p}^p = \| |u| \|^p_{\mathcal{H}^p} + \sum_{i=1}^{k} \| |D^{(i)}u|\|^p_{(\mathcal{H}^p)^{1+i}} < \infty.$
\end{itemize}
For example if a process $\zeta\in \mathbb{L}_{2,2}(\mathbb{R})$, we have 
\begin{align*}
\| \zeta \|_{\mathbb{L}_{1,2}}^2 &= \mathbb{E}\Big[ \int_{0}^{T}|\zeta_t|^2\mathrm{d}t + \int_{0}^{T}\int_{0}^{T}|D_{\theta}\zeta_t|^2\mathrm{d}\theta\mathrm{d}t  \Big],\\
\| \zeta \|_{\mathbb{L}_{2,2}}^2 &= \| \zeta \|_{\mathbb{L}_{1,2}}^2 + \mathbb{E}\Big[ \int_{0}^{T}\int_{0}^{T}\int_{0}^{T}|D_{\theta_1}D_{\theta_2}\zeta_t|^2\mathrm{d}\theta_1\mathrm{d}\theta_2\mathrm{d}t  \Big].
\end{align*}

\subsection{Some preliminary results}
Consider the BSDE 
\begin{equation}\label{eqmain1}
Y_t = \xi + \int_{t}^{T} g(s,\omega,Y_s,Z_s)\mathrm{d}s -\int_{t}^{T} Z_s\mathrm{d}B_s,
\end{equation}
and recall the following result on the Malliavin differentiablity of solution to the BSDE \eqref{eqmain1} in the Lipschitz framework (see for example \cite{KPQ97}).

\begin{thm}[Malliavin differentiability]\label{Maldiff}
Suppose $\xi \in \mathbb{D}^{1,2}$ and $g: \Omega\times [0,T]\times \mathbb{R}\times \mathbb{R}^d\rightarrow \mathbb{R}$ is continuously differentiable in $(y,z),$ with uniformly bounded derivatives. Suppose for each $(y,z)\in \mathbb{R}\times\mathbb{R}^d,$ the process $g(\cdot,y,z)$ belongs to $\mathbb{L}_{1,2}(\mathbb{R}^d)$ with Malliavin derivative denoted by $D_{\theta}g(t,y,z).$ Let $(Y,Z)$ be the solution of the associated BSDE \eqref{eqmain1} and suppose in addition
\begin{itemize}
	\item[(i)] $(g(t,0,0))_{t\in [0,T]} \in \mathcal{H}^4(\mathbb{R})$ and $\xi \in L^4(\mathbb{R}),$
	\item[(ii)] $\int_{0}^{T}\mathbb{E}[|D_{\theta}\xi|^2]\mathrm{d}\theta <\infty, \int_{0}^{T}\|(D_{\theta}g)(t,Y,Z)\|_{\mathcal{H}^2}^2\mathrm{d}\theta <\infty \text{ a.s.}$ and for $\text{a.a. } \theta\in [0,T]$ and any $t\in [0,T], (y^1,z^1),(y^2,z^2)\in \mathbb{R}\times\mathbb{R}^d$
	\[ |D_{\theta}g(t,y^1,z^1) - D_{\theta}g(t,y^2,z^2)|\leq K_{\theta}(t)(|y^1-y^2|+|z^1-z^2|)\text{  a.s.},  \]
	with $\{ K_{\theta}(t), 0\leq \theta,t\leq T \}$ a positive real valued adapted process satisfying $\int_{0}^{T}\|K_{\theta}\|_{\mathcal{H}^4}^4\mathrm{d}\theta <\infty.$
\end{itemize}
Then $(Y,Z)\in L^2(\mathbb{D}^{1,2}\times (\mathbb{D}^{1,2})^d).$ Furthermore for each $1\leq i\leq d$ a version of $\{ (D_{\theta}^iY_t, D_{\theta}^i Z_t);0\leq \theta, t\leq T  \}$ is given by
\begin{align*}
	D_{\theta}^iY_t &= 0, \quad D_{\theta}Z_t = 0, \qquad 0\leq t < \theta \leq T;\\
	D_{\theta}^i Y_t  &= D_{\theta}^i \xi  -\int_{t}^{T} D_{\theta}^i Z_s \mathrm{d}B_s\\
	& + \int_{t}^{T} [ (D_{\theta}^i g)(s,Y_s,Z_s) + \langle (\nabla g)(s,Y_s,Z_s), (D_{\theta}^i Y_s, D_{\theta}^i Z_s) \rangle  ]\mathrm{d}s \qquad \theta \leq t\leq T.
\end{align*}
Moreover $\{ D_{\theta} Y_t; 0\leq t \leq T \}$ defined by the solution to the above BSDE is a version of $\{ Z_t; 0\leq t \leq T \}.$
\end{thm}
We end this section by recalling the following result  (see \cite[Theorem 1.2.3]{Nua06})
\begin{lemm}\label{lemma 1.2.3}
Let $(F_n)_{n\geq 1}$ be a sequence of random variables in $\mathbb{D}^{1,2}$ that converges to $F$ in $L^2(\Omega)$ and such that 
\[ \sup_{n\geq 1}\mathbb{E}\left[  \|D F_n \|_{L^2} \right] < \infty .\]
Then, $F$ belongs to $\mathbb{D}^{1,2}$, and the sequence of derivatives $(DF_n)_{n\geq 1}$ converges to $DF$ in the weak topology of $L^2(\Omega\times [0,T])$.
\end{lemm}
\section{Quadratic BSDEs with nonlinearity of type $f(|y|)|z|^2$}\label{sectexun}
The main aim of this section is to study the well posedness and derive an a-priori estimates of the  BSDE \eqref{eqmain1} when the parameters $\xi$ and $g$ satisfy the following assumptions:
\begin{assum}\label{assum1}
$\xi$ is an $\mathfrak{F}_T\text{-measurable}$ uniformly bounded random variable, i.e., $\| \xi \|_{L^{\infty}} < \infty;$ 
\end{assum}
\begin{assum}\label{assum2} The function $g:  [0,T]\times \Omega\times \mathbb{R} \times \mathbb{R}^d \rightarrow \mathbb{R}$ is $\mathfrak{F}\text{-predictable}$ and continuous in its space variables. There exist $\Lambda_0,\Lambda_y,\Lambda_z >0$, and a locally bounded function $f: \mathbb{R}\mapsto \mathbb{R}_{+}$; $f\in L^1_{loc}(\mathbb{R}^+)$ such that for all $(t,\omega,y,z), (t,\omega,y',z') \in [0,T]\times \Omega \times \mathbb{R} \times \mathbb{R}^{d}$, $\alpha \in [0,1)$ $\|g(t,0,0)\|_{L^{\infty}} \leq \Lambda_0$ and
\begin{align*}
	&|g(t,\cdot,y,z)-g(t,\cdot,y',z')|\\
	& \leq \Lambda_y\Big(1 + |z|^{\alpha} + |z'|^{\alpha}\Big)|y-y'| + \Lambda_z \Big(1+  ( f(|y|)+ f(|y'|))(|z|+|z'|)\Big)|z-z'| \text{ a.s.} 
\end{align*}
\end{assum}	
\begin{remark}\label{remk}
It is readily seen that under Assumptions \ref{assum1} and \ref{assum2}, the generator $g$ is necessary of the following form:
\begin{equation}\label{qg1}
	|g(t,\cdot,y,z)| \leqslant \Lambda_0 + \Lambda_y |y|+ \Lambda_z ( |z| + f(|y|)|z|^2)  \text{ a.s.} 
\end{equation}		
Indeed, 
\begin{align*}
	|g(t,y,z)|&\leq |g(t,0,0)| + |g(t,y,0) -g(t,0,0)| + |g(t,y,z) - g(t,y,0)|\\
	&\leq \Lambda_0 + \Lambda_y |y| + \Lambda_z |z| + 2\Lambda_z f(|y|)|z|^2 \text{ a.s.}
\end{align*}
\end{remark}
Unless otherwise stated, in this paper $\varphi$ stands for the smallest continuous and increasing function such that $f(x) \leq \varphi(x)$ for all $x \in \mathbb{R}$ (compare with \cite{Bahlali1}). 

The next result concerns the existence and uniqueness of solution to the BSDE \eqref{eqmain1} under Assumptions \ref{assum1}--\ref{assum2}.

\begin{thm}[Existence and uniqueness]\label{th1}
Under Assumptions \ref{assum1}--\ref{assum2}, the BSDE \eqref{eqmain1} has a unique solution $(Y,Z) \in \mathcal{S}^{\infty}(\mathbb{R})\times\mathcal{H}^2(\mathbb{R}^d).$

\end{thm}
\begin{proof} \leavevmode
{\bf Existence:}  Since the generator $g$ is of quadratic type, using  \cite[Theorem 3.1 and Corollary 3.2]{Bahlali1}, the  equation \eqref{eqmain1} has a maximal solution $(Y,Z) \in \mathcal{S}^{\infty}\times\mathcal{H}_{\text{BMO}} $ . \\
{\bf Uniqueness:}  Let $t \in [0,T]$ and let $(Y^i,Z^i) \in \mathcal{S}^{\infty}\times\mathcal{H}_{\text{BMO}}$ be two solutions to BSDE \eqref{eqmain1} with $i=1,2$. For each $i$, the $\mathcal{S}^{\infty}\text{-norms}$ of the process $Y^i$ is uniformly bounded by the constant $\Upsilon^{(1)}$  that only depends on $\|\xi \|_{L^{\infty}}$, $\Lambda_y$ and $T$ (see the Lemma below). Now, set $\delta Y = Y^1-Y^2,$ $\delta Z = Z^1-Z^2$ then the dynamics of $(\delta Y)_t$ is given by
\[
\delta Y_t = \int_{t}^{T} \big( g(s,Y^1_s,Z^1_s) -g(s,Y^2_s,Z^2_s)\big)\mathrm{d}s -\int_{t}^{T} \delta Z_s \mathrm{d}B_s .
\]
Let $\Gamma_t, e_t$ and $ \Pi_t$ be defined by
\begin{align*}
	\Gamma_t:&= \frac{g(t,Y^1_t,Z^1_t) - g(t,Y^2_t,Z^1_t)}{Y^1_t - Y^2_t}{1}_{\{Y^1_t - Y^2_t\neq 0\}}, \quad  e_t:= \exp\Big( \int_{0}^{t} \tilde \Gamma_s \mathrm{d}s \Big),\\
	\Pi_t:&= \frac{g(t,Y^2_t,Z^1_t) - g(t,Y^2_t,Z^2_t)}{|Z^1_t - Z^2_t|^2}(Z^1_t - Z^2_t){1}_{\{|Z^1_t - Z^2_t|\neq 0\}}.
\end{align*}
From Assumption \ref{assum2}, we obtain $|{\Pi}| \leq \Lambda_z (1 + 2f(|Y^2|)(|Z^1| + |Z^2|)),$ from which we have $\|{\Pi}\|_{\mathcal{H}_{\text{BMO}}} \leq \tilde{\Lambda}:= \Lambda_z (\sqrt{T} + \varphi(\|Y^2\|_{\tiny\mathcal{S}^{\infty}}) (\|Z^1\|_{\mathcal{H}_{\text{BMO}}} + \|Z^2\|_{\mathcal{H}_{\text{BMO}}}))$, where $\varphi(\|Y^1\|_{\tiny\mathcal{S}^{\infty}}):=\sup_{0\leq y\leq \|Y^1\|_{\tiny\mathcal{S}^{\infty}}} \varphi(y) <\infty$. Thus ${\Pi}*B$ is a BMO martingale since $Z^1, Z^2 \in \mathcal{H}_{\text{BMO}}$. Hence the probability measure ${\mathbb{Q}}$ with Radon-Nykodim density $\mathrm{d}{\mathbb{Q}}/\mathrm{d}{\mathbb{P}}= \mathcal{E}(\int_{0}^{\cdot}\tilde\Pi_{\cdot}\mathrm{d}B_{\cdot})$ is well defined and the process $B_{\cdot}^{\mathbb{Q}} = B_{\cdot} - \int_{0}^{\cdot} {\Pi}_s\mathrm{d}s $ is a ${\mathbb{Q}}\text{-Brownian motion}.$ Moreover for $r> 1$, we have $\mathcal{E}(\Pi) \in L^{r}$ (see Lemma \ref{lem 2.1}). On the other hand, using Assumption \ref{assum2} once more, we deduce that $ |{\Gamma}| \leq \Lambda (1 + 2|Z^1|^{\alpha})$. This implies that $\Gamma \in \mathcal{H}_{\text{BMO}}$. Thus, the process ${e}$ is integrable (see (P4) in Lemma \ref{lem 2.1}) i.e., there is $p\geq 1$ and $\varepsilon \in (0,2)$ such that $\mathbb{E}\Big[\exp\Big( p\int_{0}^{T} |\Gamma_t|^{\varepsilon}\mathrm{d}t \Big)\Big] < \infty$. In addition for all $p\geq 1,$ $\tilde{e} \in \mathcal{S}^p(\mathbb{R})$. The Girsanov theorem and  H\"{o}lder's inequality yield: for every $r> 1$
\begin{align*}
	\mathbb{E}^{{\mathbb{Q}}}\Big[ \int_{0}^{T} |{e}_s|^2 |\delta Z_s|^2 \mathrm{d}s  \Big] &\leq \mathbb{E}\Big[ \mathcal{E}\Big(\int_{0}^{T}\Pi_{s}\mathrm{d}B_{s}\Big)\sup_{0\leq t\leq T}|{e}_t|^2 \int_{0}^{T}  |\delta Z_s|^2 \mathrm{d}s\Big]\nonumber\\
	&\leq \mathbb{E}\Big[\mathcal{E}\Big(\int_{0}^{T}\Pi_{s}\mathrm{d}B_{s}\Big)^{r}\Big]^{\frac{1}{r}}
	\mathbb{E}\Big[\sup_{0\leq t\leq T}|{e}_t|^{2q}\Big( \int_{0}^{T}  |\delta Z_s|^2 \mathrm{d}s\Big)^{r'}\Big]^{\frac{1}{r'}} < \infty,
\end{align*}
where $r'$ is the H\"{o}lder conjugate of $r$.  Thus, the stochastic integral $\int_{0}^{\cdot}{e}_s \delta Z_s\mathrm{d}B_s^{{\mathbb{Q}}}$ defines a true ${\mathbb{Q}}\text{-martingale}$.
By applying It\^{o}'s formula to the semimartingale $({e}\delta Y)_t$ under $\tilde{\mathbb{Q}}$, we obtain that
\begin{align}\label{3.5}
	{e}_t\delta Y_t + \int_{t}^{T} {e}_s \delta Z_s \mathrm{d}B_s^{{\mathbb{Q}}} = \int_{t}^{T} {e}_s[ -{\Gamma}_s\delta Y_s + g(s,Y^1_s,Z^1_s)- g(s,Y^2_s,Z^1_s)  ] \mathrm{d}s = 0.
\end{align}
It follows that ${e}_t\delta Y_t = 0\,\, {\mathbb{Q}}\text{-a.s.}$ for all $t \in [0,T]$. Thus, ${e}_t\delta Y_t = 0\,\, {\mathbb{P}}\text{-a.s.}$ for all $t \in [0,T]$ (since ${\mathbb{Q}}$ and $\mathbb{P}$ are equivalents). Consequently $\delta Y_t = 0 \,\, {\mathbb{P}}\text{-a.s.}$ for all $t\in [0,T]$ provided that ${e}_t(\omega) > 0$ for all $\omega$ outside a $\mathbb{P}\text{-negligible}$ set $A$. The later is satisfied due to the continuity of the process ${e}_t$ for all $t\in [0,T].$
Hence, $Y^1_t = Y^2_t \,\, \mathbb{P}\text{-a.s.},\forall t\in [t,T]$.
Using \eqref{3.5} and the fact that $\delta Y =0,$ we deduce from the It\^{o}'s isometry that $\mathbb{E}^{{\mathbb{Q}}}\int_{0}^{T}|{e}_s|^2|\delta Z_s|^2 \mathrm{d}s = 0,$ the latter implies that $Z^1_t = Z^2_t \,\, \mathrm{d}t\otimes\mathbb{P}\text{-a.s.}$. This conclude the proof.
\end{proof}

Below we provide a  more precise bounds for both the $\mathcal{S}^{\infty}$ and $\mathcal{H}_{\text{BMO}}$ norms of the processes $Y$ and $Z$ respectively
\begin{lemm}
Under Assumptions \ref{assum1} and \ref{assum2}, the solution $(Y,Z)$ to the BSDE \eqref{eqmain1} satisfies the following bounds:
\begin{align}
	\|Y\|_{\mathcal{S}^{\infty}} &\leq \Upsilon^{(1)} := (\|\xi \|_{L^{\infty}} +\Lambda_0 T) e^{\Lambda_y T},\label{bound1}\\
	\|Z*B\|_{BMO} &\leq \Upsilon^{(2)} \label{bound2}
\end{align} 
where, $\Upsilon^{(2)}:=  2\Upsilon^{(1)}\Big(\Upsilon^{(1)} + T(\Lambda_0 +\Lambda_z +\Lambda_y\Upsilon^{(1)})\Big) \exp(4\| (1+\Lambda_zf) \|_{L^1[0,\Upsilon^{(1)}]}).$
\end{lemm}
\begin{proof}
Let $\Pi_t$ be defined by: $ |Z_t|^2\Pi_t:= (g(t,Y_t,Z_t) - g(t,Y_t,0))Z_t{1}_{\{Z_t\neq 0\}}$. By using the same observations as in the proof of the previous theorem, we have that $\mathcal{E}(\Pi*B)$ is uniformly integrable and the process $B^{\mathbb{Q}}_{\cdot}= B_{\cdot} -\int_{0}^{\cdot}\Pi_s\mathrm{d}s $ is a $\mathbb{Q}\text{-Brownian motion}$ with the measure $\mathbb{Q}$ given by $\mathrm{d}\mathbb{Q} = \mathcal{E}(\Pi*B)\mathrm{d}\mathbb{P}$. Then,
\begin{align*}
	|Y_t|&\leq \mathbb{E}^{\mathbb{Q}}\Big( |\xi| + \int_{t}^{T} |g(s,Y_s,0)|\mathrm{d}s/\mathfrak{F}_t\Big) \leq \|\xi\|_{\tiny L^{\infty}} + \Lambda_0 T + \Lambda_y\mathbb{E}^{\mathbb{Q}}( \int_{t}^{T} |Y_s|\mathrm{d}s/\mathfrak{F}_t ).
\end{align*}
Therefore, the Gronwall's lemma yields :
$|Y_t| \leq (\|\xi\|_{\tiny L^{\infty}} + \Lambda_0 T)e^{\Lambda_y T} .$

On the other hand, for any {locally} integrable function $f_1$, we define: 
\begin{align*}  
	K(y):= \int_{0}^{y} \exp\Big(-2 \int_{0}^{z}f_1(u)\mathrm{d}u\Big)\mathrm{d}z\\
	v(x) :=  \int_{0}^{x} K(y)\exp\Big(2 \int_{0}^{y}f_1(u)\mathrm{d}u\Big)\mathrm{d}y. 
\end{align*}
It's readily seen that $ v\in W_{\text{loc}}^{1,2}(\mathbb{R})$ and satisfies almost everywhere: $1/2v''(x)-f_1(x)v'(x) = 1/2$ (see \cite{Bahlali1}). Moreover for any $R >0$ such that $|x|\leq R$, we have $|v(x)|\leq R^2\exp(4\| f_1\|_{L^1[0,\Upsilon^{(1)}]})$ and $|v'(x)| \leq  R\exp(4\| f_1\|_{L^1[0,\Upsilon^{(1)}]})$. Recall from Remark \ref{remk} that the driver $g$ satisfies $|g(t,y,z)|\leq \Lambda + \Lambda_y |y| + f_1(|y|)|z|^2$, where $\Lambda= \Lambda_0 + \Lambda_z$ and $f_1(|y|):= \Lambda_z (1+ f(|y|)).$ Then, using the It\^{o}-Krylov formula for BSDE (see \cite[Theorem 2.1]{Ouknine}) we obtain that
\begin{align*}
	v(|Y_{\tau}|) &= v(|Y_T|) + \int_{\tau}^{T} \sgn(Y_r)v'(|Y_u|)g(u,Y_u,Z_u)\mathrm{d}u - \frac{1}{2}\int_{\tau}^{T} v''(|Y_u|)|Z_u|^2\mathrm{d}u + M_{\tau}^T\notag\\
	&\leq v(|Y_T|) -\frac{1}{2} \int_{\tau}^{T} |Z_u|^2\mathrm{d}u + \int_{\tau}^{T}(\Lambda + \Lambda_y|Y_u|)v'(|Y_u|)\mathrm{d}u +  M_{\tau}^T
\end{align*}
for any stopping time $\tau$. Here $ M_{\tau}^T$ represents the martingale part. By taking the conditional  expectation with respect to $\mathfrak{F}_{\tau}$, we deduce that
\begin{align*}
	\frac{1}{2} \mathbb{E}\Big(\int_{\tau}^{T} |Z_u|^2\mathrm{d}u/\mathfrak{F}_{\tau}\Big)	&\leq \mathbb{E}\Big( v(|Y_T|)  + \int_{\tau}^{T}(\Lambda + \Lambda_y|Y_u|)v'(|Y_u|)\mathrm{d}u/ \mathfrak{F}_{\tau} \Big).
\end{align*}
The result is obtained from the bounds of $Y$, $v$ and $v'$. This completes the proof.
\end{proof}
\begin{remark} \label{remmaldif1}\leavevmode
Note that the bound in \eqref{bound1} does not depend on $\alpha$, $\Lambda_z$, thus not on the $\mathcal{H}_{\text{BMO}}\text{-norm}$ of the control process $Z$. In addition, if $f\equiv 0,$ i.e., the driver $g$ is Lipschitz in $z$ and still stochastic Lipschitz in $y$, the BSDE \eqref{eqmain1} has a unique solution $(Y,Z) \in \mathcal{S}^{\infty}(\mathbb{R})\times \mathcal{H}^2(\mathbb{R}^d)$ such that the bound of $Y$ is given by \eqref{bound1}.
\end{remark}
In the next lemmas, we give main (a-priori) estimates of this paper. They will be extensively used to establish both the classical and variational differentiability of solution the $(Y,Z)$ to the BSDE \eqref{Main}. The proof of Lemmas \ref{thap} and \ref{thap2} are deported to Appendix \ref{proofs}. 
\begin{lemm}\label{thap0}
Let $(Y, Z) \in \mathcal{S}^{\infty}\times \mathcal{H}_{\text{BMO}}$ be the solution to BSDE \eqref{eqmain1}, with coefficient $(\xi, g)$ satisfying Assumptions \ref{assum1}--\ref{assum2}. Assuming further that for any $\gamma\geq 1$ we have $\int_0^T |g(s,0,0)|\mathrm{d}s \in L^{\gamma}$. Then for $p > 1$, there exists $q \in (1, \infty)$ only depending on, $T$, $r$ and the BMO norm of $Z*B$ such that
\begin{align*}
	\mathbb{E}\Big[ \sup_{t\in [0,T]}| Y_t|^{2p} \Big] + \mathbb{E}\left[ \Big(\int_0^T |Z_t|^2\mathrm{d}t \Big)^p \right] \leq C\mathbb{E}\Big[ |\xi|^{2pq} + \Big(\int_{0}^{T}| g(s,0,0)|\mathrm{d}s \Big)^{2pq}  \Big]^{\frac{1}{q}}	.
\end{align*}
\end{lemm}
The above Lemma, provides the moments estimates for BSDEs with drivers satisfying a stochastic Lipschitz condition both on their second and third component. Its proof is omitted.
\begin{lemm}[A priori estimates]\label{thap}
Let $(Y^{i}, Z^{i}) \in \mathcal{S}^{\infty}\times \mathcal{H}_{\text{BMO}}$ be the solution to BSDE \eqref{eqmain1}, with terminal value $\xi^i$ and generator $g^i$ for $i\in \{1,2 \}$ satisfying Assumptions \ref{assum1}-\ref{assum2}. Then for $p > 1$, there exists $q \in (1, \infty)$ only depending on, $T$, $r$ and the BMO norm of $Z*B$ such that
\begin{align*}
	&\| Y^1 - Y^2 \|_{\mathcal{S}^{2p}}^{2p} + \|Z^1 - Z^2\|_{\mathcal{H}^{2p}}^{2p} \\
	&\leq C\mathbb{E}\Big[ |\xi_1 -\xi_2|^{2pq} + \Big(\int_{0}^{T}| g^1(s,Y^2_s,Z^2_s) - g^2(s,Y^2_s,Z^2_s)|\mathrm{d}s \Big)^{2pq}  \Big]^{\frac{1}{q}}	.
\end{align*}
\end{lemm}

We also deduce the following stability result.
\begin{cor}\label{stab}
Let $\{\xi_n\}_{n\in \mathbb{N}}$ be a sequence of bounded $\mathfrak{F}_T\text{-adapted}$ random variables that converges $\mathbb{P}\text{-a.s.}$ to $\xi$. Let $(g_n)_{n\in \mathbb{N}}$ be a sequence of drivers satisfying Assumption \ref{assum2} with the same constant $\Lambda$ and the same function $(f)$  , such that for $\mathrm{d}t\otimes \mathrm{d}\mathbb{P}\text{-a.s.}$ $(t,\omega) \in [0,T]\times \Omega$
$(g_n)_{n\in \mathbb{N}}(t,\omega,y,z)$ converges to $(g)(t,\omega,y,z)$ locally uniformly in $(y,z)\in \mathbb{R}\times \mathbb{R}^d$. Let  $(Y^n,Z^n) \in \mathcal{S}^{\infty}(\mathbb{R})\times \mathcal{H}^{2}(\mathbb{R}^d)$ be the solution to the BSDE \eqref{eqmain1} with parameters $(\xi_n,g_n)$. Then the BSDE \eqref{eqmain1} has a unique solution $(Y,Z) \in \mathcal{S}^{\infty}(\mathbb{R})\times \mathcal{H}^{2}(\mathbb{R}^d)$ such that $\mathbb{P}\text{-a.s.}$ $Y^n_t$ converges to $Y_t$ uniformly in $t \in [0,T]$ and $Z^n$ converges to $Z$ in $\mathcal{H}^{2}(\mathbb{R}^d)$  
\end{cor}
The proof of Corollary \ref{stab} is straightforward from Lemma \ref{thap}, again we do not provide it here. 

In the subsequent Lemma, we suggests a $L_p-L_{2p}$ type a-priori estimates of solutions to the BSDE \eqref{eqmain1} when the coefficients satisfy Assumptions \ref{assum1}--\ref{assum2}. It can be viewed as a refine version of Lemma \ref{thap} and an extension of Lemma 5.26 in \cite{GeissYlinen} for more general type of quadratic BSDEs with stochastic Lipschitz coefficients.
\begin{lemm}\label{thap2}
Let $(Y^{i}, Z^{i}) \in \mathcal{S}^{\infty}\times \mathcal{H}_{\text{BMO}}$ be the solution to the BSDE\eqref{eqmain1}, with terminal value $\xi^i$ and generator $g^i, \,i\in \{1,2 \}$ satisfying Assumptions \ref{assum1} and \ref{assum2}. Then for $p > 1$, with $p > p_0$  ($p_0$ is such that $\mathcal{E}(Z*B) \in L^{p_0}$)
\begin{align*}
	\| Y^1 - Y^2 \|_{\mathcal{S}^{p}}^{p}  &\leq C\mathbb{E}\Big[ |\xi_1 -\xi_2|^{2p} + \Big(\int_{t}^{T}| g^1(s,Y^2_s,Z^2_s) - g^2(s,Y^2_s,Z^2_s)|\mathrm{d}s \Big)^{2p}  \Big]^{\frac{1}{2}},\\
	\|Z^1 - Z^2\|_{\mathcal{H}^{p}}^{p}  &\leq C\mathbb{E}\Big[ |\xi_1 -\xi_2|^{p} + \Big(\int_{t}^{T}| g^1(s,Y^2_s,Z^2_s) - g^2(s,Y^2_s,Z^2_s)|\mathrm{d}s \Big)^{p}   \Big]	+ 	\| Y^1 - Y^2 \|_{\mathcal{S}^{p}}^{p} .
\end{align*}
\end{lemm}
\begin{proof}
	See Appendix \ref{proofs}.
\end{proof}
\section{Differentiability of parametrised quadratic BSDEs}\label{sectdiff}
In this section we discuss the both the Malliavin and the classical differentiablility of the solutions to the BSDEs \eqref{eqmain1}. We provide below, sufficient conditions for these differentiabilities  to hold. 

\subsection{Malliavin Differentiability}\label{Maldif}
In this subsection we show under some weak conditions on the generator that the solution $(Y,Z)$ to the BSDE \eqref{eqmain1} is Malliavin differentiable and the Malliavin derivatives $(D_uY_t,D_uZ_t)_{u,t\in [0,T]}$ are given as solution to a BSDE. In addition to Assumptions \ref{assum1}-\ref{assum2}, we will suppose the following additional assumptions
\begin{assum}\label{assum5}\leavevmode
\begin{itemize}
	\item[(M1)] 
	The random variable $\xi$ belongs to $\mathbb{D}^{1,\infty}.$ The function $g:[0,T]\times\Omega\times \mathbb{R}\times \mathbb{R}^d \mapsto \mathbb{R}$ is adapted, measurable and continuously differentiable in $(y,z)$.
\item[(M2)] For each $(y,z) \in \mathbb{R}\times \mathbb{R}^d,$ it holds that  $(g(t,y,z))_{t \in [0,T]} \in \mathbb{L}_{1,2p}(\mathbb{R})$ for all $p \geq 1$. Its Malliavin derivative denoted by $(D_ug(t,y,z))_{u,t \in [0,T]}$ satisfies
\begin{align*}
	|D_u g(t,y,z)| \leq K_u(t)(1 + |y| + f(|y|) |z|^{\alpha}) + \tilde{K}_u(t)( 1+ |z|^{\alpha} +f(|y|)|z|) \text{ a.s.}
\end{align*}
for any $(u,t,y,z) \in [0,T]\times [0,T]\times \mathbb{R}\times \mathbb{R}^d,$ $\alpha\in (0,1)$. Here $(K_u(t))_{u,t \in [0,T]}$ and $(\tilde K_u(t))_{u,t \in [0,T]}$ are two positive adapted processes such that  for all $p\geq 1$
\begin{align*}
	\sup_{0\leq t \leq T}	\int_{0}^{T} \mathbb{E} | K_u(t)|^{2p} \mathrm{d}u+ \| \tilde K_u(t)\|^{2p}_{\mathcal{S}^{2p}} < \infty.
\end{align*}

\end{itemize}
\end{assum}
Below, we state the main result in this subsection:
\begin{thm}\label{Main2}
Suppose $g$ and $\xi$ satisfy Assumptions \ref{assum1}--\ref{assum2}, and \ref{assum5}. Then the solution process $(Y,Z)$ to the BSDE \eqref{eqmain1} is in $\mathbb{L}_{1,2} \times (\mathbb{L}_{1,2})^d$ and a version of $(D_uY_t,D_uZ_t)_{u,t\in [0,T]}$ is the unique solution to
\begin{align}\label{MalY}
D_u Y_t &= 0 \text{ and } D_u Z_t = 0, \text{ if } t\in [0,u),\nonumber\\
D_u Y_t & = D_u\xi -\int_{t}^{T} D_u Z_s\mathrm{d}B_s\nonumber\\
& \qquad + \int_{t}^{T} \left[ (D_ug)(s,\Theta_s) + \langle (\nabla g)(s,\Theta_s), D_u\Theta_s  \rangle   \right]\mathrm{d}s, \text{ if } t\in [u,T].
\end{align}
Moreover, $\{ D_tY_t: 0\leq t \leq T \}$ is a version of $\{ Z_t: 0\leq t \leq T\}.$
\end{thm}

The strategy to prove the above theorem is also well known and analogous to \cite{ArImDR, ImkDosReis, ImkDos, DosReis}. It is performed in two main steps. We first build a family of truncated BSDEs that approximate the BSDE \eqref{eqmain1} and then, prove some uniform bounds for solutions to the truncated BSDEs. Second we apply  a  compactness result (Lemma \ref{lemma 1.2.3}) to derive the desired result.

\subsubsection{Family of truncated generators}:
Let us consider the family $(\tilde{\rho}_n)_{n\in \mathbb{N}}$ of smooth (continuously differentiable) real valued functions,  that truncated the identity on the real line. We use this family of functions to truncate the variables $y$ and $z$ simultaneously in the driver $g(\cdot,y,z),$ for $(y,z) \in \mathbb{R}\times\mathbb{R}^d$. If the procedure is well known (see for example \cite{ArImDR, ImkDos, DosReis}), it is worth mentioning that truncating the two variables at the same time is not so common in the literature. This approach is motivated by the  form of the driver which is not uniformly Lipschitz in $y$. 

The family of functions $(\tilde{\rho}_n)_{n\in \mathbb{N}}$ are such that:
\begin{itemize}
\item[(a)] $(\tilde{\rho}_n)_{n\in \mathbb{N}}$ converges uniformly to the identity. For all $n \in \mathbb{N}$ and $x \in \mathbb{R}$ 
\begin{equation}\label{trunc}
\tilde{\rho}_n(x) =	\begin{cases}
	n+1,& \quad x > n+2,\\
	x,& \quad |x|\leq n,\\
	-(n+1),& \quad x < -(n+2).
\end{cases}
\end{equation}
In addition $|\tilde{\rho}_n(x)| \leq |x|$ and $ |\tilde{\rho}_n(x)| \leq n+1 $.
\item[(b)] The derivative $\nabla \tilde{\rho}_n$ is absolutely uniformly bounded by 1, and converges to 1 locally uniformly.
\end{itemize}

Let $(g_n)_{n\in \mathbb{N}}$ be the sequence defined by
\begin{equation}\label{eqmain3-} 
g_n(t,y,z) := g(t,\tilde{\rho}_n(y),\rho_n(z)) \text{ for } (\omega,t,y,z) \in \Omega\times [0,T]\times \mathbb{R}\times \mathbb{R}^d,\,\, n\in \mathbb{N}, 
\end{equation}
where $\rho_n: \mathbb{R}^d \rightarrow \mathbb{R}^d, z\mapsto \rho_n(z) = (\tilde{\rho}_n(z_1),\ldots,\tilde{\rho}_n(z_d)), n\in \mathbb{N}.$

Let us consider the following sequence $(Y^n,Z^n)_{n\geq 1}$ satisfying the BSDE 
\begin{equation}\label{eqmain3}
Y_t^n = \xi + \int_{t}^{T} g_n(s, Y_s^n, Z_s^n)\mathrm{d}s - \int_{t}^{T} Z_s^n \mathrm{d}B_s, \qquad t\in [0,T], n\in \mathbb{N}.
\end{equation}
Fix $n \in \mathbb{N}$. Using  (M1), the family $(g_n)_{n \in \mathbb{N}}$ satisfies: 
for  $(t,y,z)\in [0,T]\times \mathbb{R}\times \mathbb{R}^d$
\begin{align*} 
|\nabla_y g_n(t,y,z)| + |\nabla_z g_n(t,y,z)| \leq \Lambda_y ( 1 + (n+1)^{\alpha} )+ \Lambda_z ( 1 + (n+1)f(n+1)).
\end{align*}
In addition, using the mean value theorem, we obtain for all $t \in [0,T],$ $y,y' \in \mathbb{R},$ and $z,z' \in \mathbb{R}^d$ 
$|g_n(t,y,z) - g_n(t,y',z')|\leq \Lambda_n (|y-y'| + |z-z'|)$. Thus for each $n \in \mathbb{N},$ the family of functions $(g_n)_{n \in \mathbb{N}}$ is Lipschitz continuous in its spatial variables.


\subsubsection{Uniform bounds for solutions to BSDE \eqref{eqmain3} } Here, we will prove that the solution $(Y^n,Z^n)$ to the BSDE \eqref{eqmain3} is Malliavin differentiable with uniformly bounded $\mathbb{L}_{1,2} \times (\mathbb{L}_{1,2})^d\text{-norm}$. Recall that $\mathbb{L}_{1,2}$ stands for the set of progressively measurable processes $(\eta_t)_{0\leq t\leq T}$ which are Malliavin differentiable, with $(D_s\eta_t)_{s\leq t\leq T}$ having a progressively measurable version  and such that $\|\eta\|_{1,2}^2:= \mathbb{E}\Big[ \int_{0}^{T} |\eta_t|\mathrm{d}t + \int_{0}^{T}\int_{0}^{T} |D_s\eta_t|^2\mathrm{d}s\mathrm{d}t   \Big]<+\infty$.

We first establish the next Lemma, which provides the uniform bounds of solutions to the BSDE \eqref{eqmain3} in the Banach space $\mathcal{S}^{\infty}(\mathbb{R})\times \mathcal{H}^2(\mathbb{R}^d)$. 
\begin{lemm}\label{lemma 5.1}
For each $n \in \mathbb{N}$, the BSDE \eqref{eqmain3} admits a unique solution $(Y^n,Z^n)$ which is uniformly bounded in $\mathcal{S}^{\infty}(\mathbb{R})\times \mathcal{H}^2(\mathbb{R}^d)$. In addition, the process $Z^n \in \mathcal{H}_{BMO}$, and $\sup_{n\in \mathbb{N}}\|\mathcal{E}(Z^n* B) \|_{BMO} \leq \Upsilon^{(2)},$ where $\Upsilon^{(2)}$ is given in Theorem \ref{th1}. Furthermore, there exists $r > 1$ independent of $n$ such that $\sup_{n\in \mathbb{N}}\|\mathcal{E}(Z^n* B) \|_{ L^{r} } < \infty.$
\end{lemm}
\begin{proof}
The existence and uniqueness of solution to BSDE \eqref{eqmain3} follow directly from standard results in the theory of BSDEs, since the generator $g_n$ satisfies for each $n \in \mathbb{N}$ a Lipschitz condition in the space variables.
Let us also remark that the function $g_n$ satisfies uniformly Assumption \ref{assum2}. Indeed, for all $(t,y,z), (t,y',z') \in [0,T] \times \mathbb{R}\times \mathbb{R}^d,$ we obtain by using the properties satisfied by $\bar{\rho}_n$ and the function $f$:
\begin{align*}
&|g_n(t,y,z)-g_n(t,y',z')| \\
&= |g(t,\bar{\rho}_n(y), \rho_n(z)) - g(t,\bar{\rho}_n(y'), \rho_n(z'))|\\
&\leq \Lambda_y \left( 1+ |\rho_n(z)|^{\alpha} + |\rho_n(z')|^{\alpha} \right)|\bar{\rho}_n(y)-\bar{\rho}_n(y')| \\
&\quad+ \Lambda_z \left( 1 + (f(|\bar{\rho}_n(y)|) + f(|\bar{\rho}_n(y')|))(|\rho_n(z)| + |\rho_n(z)|)  \right)|\rho_n(z) - \rho_n(z')|\\
&\leq \Lambda_y (1+ |z|^{\alpha} + |z'|^{\alpha}) |y-y'| + \Lambda_z (1 + (f(|y|) +f(|y'|)(|z|+|z'|))|z-z'|
\end{align*} 

Moreover, a similar condition to \eqref{qg1} is also obtained almost surely:
\begin{align*} 
|g_n(t,y,z)| &\leq \Lambda_0 + \Lambda_y|\tilde{\rho}_n(y)| + \Lambda_z( |\rho_n(z)| + f(|\tilde{\rho}_n(y)|)|\rho_n(z)|^2 )\\
&\leq \Lambda_0 + \Lambda_y|y| + \Lambda_z( |z| + f(|y|)|z|^2).
\end{align*}
Hence conditions on the existence of a unique solution to the BSDE \eqref{eqmain3} in Theorem \ref{th1} are satisfied. Therefore, from \eqref{bound2} we deduce that: $\sup_{n\in \mathbb{N}}\|\mathcal{E}(Z^n* B)\|_{BMO} <\infty.$ From (P2) in Lemma \ref{lem 2.1} we then deduce the existence of such a constant $r$ for which the bound $\sup_{n\in \mathbb{N}}\|\mathcal{E}(Z^n* B) \|_{ L^{r} } < \infty$ holds. This concludes the proof.
\end{proof}
{\begin{remark}
As a consequence of Corollary \ref{stab},  the sequence $(Y^n)_{n\in \mathbb{N}}$ converges to $Y$ uniformly on $[0,T],$ the sequence $(Z^n)_{n\in \mathbb{N}}$ converges to $Z$ in $\mathcal{H}^2(\mathbb{R}^d)$ and $(Y,Z)$ solves the BSDE \eqref{eqmain1}.
\end{remark}}

\begin{lemm}\label{lemma 5.2}
Suppose $\xi \in\mathbb{D}^{1,\infty}$ and for each $n \in \mathbb{N},$ let  $g_n$ be as in \eqref{eqmain3-}. Then the solution $\Theta^n =(Y^n,Z^n)_{n\in \mathbb{N}}$ to the BSDE \eqref{eqmain3} belongs to $\mathbb{L}_{1,2} \times (\mathbb{L}_{1,2})^d$. A version of $\{ (D_u Y^n_t, D_u Z_t^n), 0\leq u,t\leq T  \}$ is given by 
\begin{align}\label{MalY^n}
D_u Y_t^n =&0 \text{ and } D_u Z_t^n = 0, \text{ if } t\in [0,u),\nonumber\\
D_u Y_t^n =& D_u\xi -\int_{t}^{T} D_u Z_s^n\mathrm{d}B_s\\
& + \int_{t}^{T} \left[ (D_ug_n)(s,\Theta_s^n) + \langle (\nabla g_n)(s,\Theta_s^n), D_u\Theta_s^n  \rangle   \right]\mathrm{d}s, \text{ if } t\in [u,T].\nonumber
\end{align}
Moreover $\{ D_tY_t^n, 0\leq t \leq T \}$ defined by the above equation is a version of $\{ Z_t^n, 0\leq t \leq T\}.$
Furthermore for any $p > 1,$ the following holds:
\begin{align}\label{eqboundmalder1}
\sup_{n\in \mathbb{N}}\int_{0}^{T} \mathbb{E}\Big[ \|D_u Y^n \|_{\mathcal{S}^{2p} }^{2p} + \|D_u Z^n \|_{\mathcal{H}^{2p} }^{2p} \Big]\mathrm{d}u < \infty.
\end{align}
\end{lemm}

\begin{proof}
The proof of the first statement concerning the Malliavin derivatives $(DY^n,DZ^n)$ of $\Theta^n =(Y^n,Z^n)_{n\in \mathbb{N}}$ and the representation \eqref{MalY^n} follows from Theorem \ref{Maldiff} under  Assumptions \ref{assum5}.

Let us now focus on the proof of the bound \eqref{eqboundmalder1}. Note that, the generators $F^n$ in equation \eqref{MalY^n} is stochastic Lipschitz in its second and third variables. Indeed, from the properties of the functions $\rho_, \tilde{\rho}_n$ and $f$, one can obtain that:
\begin{align*}
|\nabla_y g_n(t,Y^n_t,Z^n_t)| &\leq \Lambda_y(1+ |Z^n_t|^{\alpha})\\
|\nabla_z g_n(t,Y^n_t,Z^n_t)| &\leq \Lambda_z(1+ f(|Y^n_t|)|Z^n_t|).
\end{align*}
Then, Lemma \ref{lemma 5.1} implies that: $\nabla_y g_n *B := \int_0^{\cdot} \nabla_y g_n(s,\Theta_s)\mathrm{d}B_s$ and $\nabla_z g_n *B := \int_0^{\cdot} \nabla_z g_n(s,\Theta_s)\mathrm{d}B_s$ are both BMO martingales, uniformly in $n \in \mathbb{N}.$ 
Therefore, the subsequent bound follows from Lemma \ref{thap0}: for $p>1$, there exists $q\in (1,\infty)$ such that
\begin{align*}
\mathbb{E}\Big[ \sup_{t\in [0,T]}|D_u Y_t^n|^{2p} + \Big( \int_{0}^{T} |D_u Z_s^n|^2 \mathrm{d}s\Big)^p \Big]
&\leq C \mathbb{E}\Big( |D_u\xi|^{2pq} + \Big\{\int_{0}^{T} |(D_u g_n)(t,Y_t^n,Z_t^n)| \mathrm{d}t\Big\}^{2pq}   \Big)^{\frac{1}{q}},
\end{align*} 
and $q$ only depends on $T,p$ and $\|Z^n*B\|_{BMO}.$ Hence, the Jensen's inequality for concave functions leads to:
\begin{align*}
&	\int_{0}^{T}\Big\{ \mathbb{E}\Big[ \sup_{t\in [0,T]}|D_u Y_t^n|^{2p} + \Big( \int_{0}^{T} |D_u Z_s^n|^2 \mathrm{d}s \Big)^p \Big] \Big\}\mathrm{d}u\\
\leq &C \Big( \int_{0}^{T} \mathbb{E}\Big[ |D_u\xi|^{2pq} + \Big( \int_{0}^{T} |(D_u g_n)(t,Y_t^n,Z_t^n)| \mathrm{d}t \Big)^{2pq} \Big] \mathrm{d}u  \Big)^{\frac{1}{q}}.
\end{align*} 
Since $\xi\in\mathbb{D}^{1,\infty}$, it follows that the first term on the right side of the above inequality is finite. As for the second term is also finite,  set $\gamma= 2pq$ and recall that the sequence of functions $(g_n)_{n\in \mathbb{N}}$ is given by equation \eqref{eqmain3-}. From assumption (M2), we deduce 
\begin{align*}
&\mathbb{E} \Big( \int_{0}^{T} |(D_u g_n)(t,Y_t^n,Z_t^n)| \mathrm{d}t \Big)^{\gamma}\\
\leq& C \mathbb{E} \Big( \int_{0}^{T} K_u(t)(1+|\tilde{\rho}(Y_t^n)|+ f(|\tilde{\rho}(Y_t^n)|)|{\rho}(Z_t^n)|^{\alpha} )\mathrm{d}t \notag\\
&+ \sup_{t\in [0,T]}|\tilde{K}_u(t)|^2
+ \Big\{\int_{0}^{T} (1 + |{\rho}(Z_t^n)|^{\alpha} + f(|\tilde{\rho}(Y_t^n)|)|{\rho}(Z_t^n)|) \mathrm{d}t\Big\}^{2} \Big)^{\gamma}
\end{align*}
Using the properties of $\tilde{\rho}\text{ and } \rho$, the fact that $f$ is increasing, H\"{o}lder's inequality and the inequality $|ab|\leq a^2 + b^2$, we deduce the following: 
\begin{align*}
&\mathbb{E} \left( \int_{0}^{T} |(D_u g_n)(t,Y_t^n,Z_t^n)| \mathrm{d}t \right)^{\gamma}\\
\leq &C \mathbb{E} \Big( \int_{0}^{T} K_u(t)(1+|Y_t^n|+ f(|Y_t^n|)|Z_t^n|^{\alpha} )\mathrm{d}t  \\
&\qquad + \sup_{t\in [0,T]}|\tilde{K}_u(t)|^2 + \Big( \int_{0}^{T} (1 + |Z_t^n|^{\alpha} + f(|Y_t^n|)|Z_t^n|) \mathrm{d}t\Big)^{2} \Big)^{\gamma}\\
\leq& C \mathbb{E}\Big(  \int_{0}^{T} |K_u(t)|^{2} \mathrm{d}t +  \int_{0}^{T} (1+|Y_t^n|+ [f(|Y_t^n|)|Z_t^n|]^{\alpha} )^2\mathrm{d}t \\
&\qquad+ \sup_{t\in [0,T]}|\tilde{K}_u(t)|^{2}
+ \Big(\int_{0}^{T} (1 + |Z_t^n|^{\alpha} + f(|Y_t^n|)|Z_t^n|)\mathrm{d}t\Big)^{2} \Big)^{\gamma} < \infty.
\end{align*}
The proof will be complete, if we prove that
\begin{align}\label{last}
\sup_{n\in \mathbb{N}}\mathbb{E}\left( \int_{0}^{T} (1+|Y_t^n|+ f(|Y_t^n|)|Z_t^n|^{\alpha} )^2\mathrm{d}t + \Big(\int_{0}^{T} (1 + |Z_t^n|^{\alpha} + f(|Y_t^n|)|Z_t^n|)\mathrm{d}t\Big)^{2}\right)^{\gamma} < \infty. 
\end{align}
Let us notice that, for any $(y,z) \in \mathbb{R}\times \mathbb{R}^d$ and $\alpha \in (0,1)$  we have the following elementary inequality:
\[ f(|y|)^2|z|^{2\alpha} \leq (1-\alpha) \big( f(|y|)\big)^{\frac{2}{1-\alpha}} + \alpha |z|^2 \leq  (1-\alpha) \big( \sup_{0\leq y\leq \|y\|}\varphi(y)\big)^{\frac{2}{1-\alpha}} + \alpha |z|^2 . \]
Then the first summand \eqref{last} writes 
\begin{align*}
&\sup_{n\in \mathbb{N}}\mathbb{E}\left( \int_{0}^{T} (1+|Y_t^n|+ f(|Y_t^n|)|Z_t^n|^{\alpha} )^2\mathrm{d}t	\right)^{\gamma}\\
& \leq C T + \sup_{n\in \mathbb{N}}\mathbb{E}\left(\sup_{t \in [0,T]} |Y_t^n|^2 + \int_{0}^{T}  |Z_t^n|^2\mathrm{d}t \right)^{\gamma}\\
& \leq C T + C \sup_{n\in \mathbb{N}}\mathbb{E}\Big[ |\xi|^{2\gamma q} + \Big(\int_{0}^{T}| g_n(s,0,0)|\mathrm{d}s \Big)^{2\gamma q}  \Big]^{\frac{1}{q}} < \infty,
\end{align*}
where, we used the Lemma \ref{thap0} again to justify the last inequality. Using almost the same technique as previously, one can derive a uniform bound for the last summand in \eqref{last}.
\end{proof}
We are now in position to prove the main theorem of this section.	
\begin{proof}[Proof of Theorem \ref{Main2} ] Let us define by $\mu$ the measure $\mathrm{d}\mu = \mathrm{d}\mathbb{P}\otimes\mathrm{d}u\otimes\mathrm{d}t$. Using Lemmas \ref{lemma 5.1} and \ref{lemma 5.2}, there exists a subsequence $(D_uY^n_t,D_uZ^n_t)$ (still indexed by $n$) that converges weakly to a limit process denoted by $(U_{u,t},V_{u,t}), 0\leq u,t\leq T$ in  the space of random variables with values in $L^2(\Omega\times[0,T])$. Thus, it is readily seen that for almost $t\in [0,T]$, the solution $(Y_t,Z_t)$ to the BSDE \eqref{eqmain1} is Malliavin differentiable and $(D_uY_t,D_uZ_t)= (U_{u,t},V_{u,t})$  $\mathrm{d}\mu\text{-a.e.}$ in $\Omega\times[0,T]\times [0,T]$. To conclude, we only need to prove that each term in equation \eqref{MalY^n} converges to its corresponding counterpart in equation \eqref{MalY} when $n$ goes to infinity. The convergence off the stochastic intergal is well known and we do not reproduce its proof here (see \cite{ArImDR} or \cite[Theorem 3.2.3]{DosReis}). 

Using assumption (M1) and the dominated convergence theorem, one can show that $
\int_0^T\langle (\nabla g_n)(s,\Theta_s^n),D_u\Theta_s^n \rangle \mathrm{d}s \text{ converges to }\int_0^T \langle (\nabla g)(s,\Theta_s),D_u\Theta_s \rangle \mathrm{d}s 
$
in the weak topology of  $L^1(\Omega\times [0,T])$. Indeed, let $\zeta$ be any bounded $\mathfrak{F}_T\text{-adapted}$ random variable. For $n \in \mathbb{N}$ and for almost all $u \in [0,T]$, using H\"older inequality, we have 
\begin{align*}
&\mathbb{E}\Big[ \zeta\int_{0}^{T} D_uY_s^n \big(  2 + |Z_s|^{\alpha} + |Z_s^n|^{\alpha}  \big)\mathrm{d}s\Big]\\
&\leq C\essup |\zeta|\mathbb{E} \Big[  \sup_{s\in [0,T]}|D_uY_s^n|^2 + \Big(\int_{0}^{T} \big(  4 + |Z_s|^{2\alpha} + |Z_s^n|^{2\alpha}  \big) \mathrm{d}s \Big)\Big].
\end{align*}
Using Lemma \ref{lemma 5.1} and \ref{lemma 5.2} and Theorem \ref{th1}, it follows that the right side of the above inequality is uniformly bounded in $n$. In addition, the continuity of $\nabla_y g$, the uniform convergence of $(Y^n)_{n\in \mathbb{N}}$ to $Y$ and the $\mathcal{H}^2$ convergence of  $(Z^n)_{n\in \mathbb{N}}$ to $Z$ yield the convergence of  $(\nabla_yg_n)(s,\Theta_s^n)$ to $(\nabla_yg)(s,\Theta_s)$ in $\mathcal{H}^2(\Omega\times[0,T])$ as $n$ goes to infinity.  Since $\zeta$ is chosen arbitrarily, the uniform convergence of the process $(DY^n)_{n\in \mathbb{N}}$ to $DY$ and the dominated convergence imply that 
\[  \lim_{n\rightarrow \infty} \int_{0}^{T} \nabla_y g_n(s,\Theta^n_s)D_uY^n_s \mathrm{d}s = \int_{0}^{T} \nabla_y g(s,\Theta_s)D_uY_s \mathrm{d}s .\]
Applying the same reasoning as above, we deduce that 
\[  \lim_{n\rightarrow \infty} \int_{0}^{T} \nabla_z g_n(s,\Theta^n_s)D_uZ^n_s \mathrm{d}s = \int_{0}^{T} \nabla_z g(s,\Theta_s)D_uZ_s \mathrm{d}s .\]

Let us now show that 
$$
D_ug_n(s,Y_s^n,Z_s^n) \text{ converges to } D_ug(s,Y_s,Z_s) \text{ in the weak topology of } L^2(\Omega\times [0,T]).
$$
Using \eqref{eqmain3-}, assumption (M2) and the assumptions on $\rho, \tilde{\rho}$ and the function $f$, we have
\begin{align*}
&\mathbb{E}\int_{0}^{T}\int_{0}^{T} |D_ug_n(s,Y_s^n,Z_s^n)|^2\mathrm{d}s\mathrm{d}u\\ \leq& \mathbb{E}\int_{0}^{T}\int_{0}^{T} |K_u(s)|^2 (1+ |\tilde{\rho}(Y_s^n)| + [f(|\tilde{\rho}(Y_s^n)|)|\rho(Z_s^n)|]^{\alpha})^2\mathrm{d}s\mathrm{d}u \\
& + \mathbb{E}\int_{0}^{T}\int_{0}^{T} |\tilde K_u(s)|^2 (1+ |\rho(Z_s^n)|^{\alpha} + f(|\tilde{\rho}(Y_s^n)|)|\rho(Z_s^n)|)^2\mathrm{d}s\mathrm{d}u\\
\leq &\mathbb{E}\int_{0}^{T}\int_{0}^{T} |K_u(s)|^2 \Big(1+ |Y_s^n| + f(|Y_s^n|)|Z_s^n|^{\alpha} \Big)^2\mathrm{d}s\mathrm{d}u \\
&+ \mathbb{E}\int_{0}^{T} \sup_{0\leq s \leq T}|\tilde K_u(s)|^2 \int_{0}^{T} \Big(1+ |Z_s^n|^{\alpha} + f(|Y_s^n|)|Z_s^n\Big)^2\mathrm{d}s\mathrm{d}u \\
\leq & C \left[ \mathbb{E}\Big( \int_{0}^{T} |K_u(s)|^2 \mathrm{d}u \Big)^2 \right]^{\frac{1}{2}} \left[ \mathbb{E} \Big( \int_{0}^{T}  (1+ |Y_s^n|^2 + \sup_{0\leq y \leq \Upsilon^{(1)}}\big(\varphi(y)\big)^{\frac{2}{1-\alpha}} + |Z_s^n|^2)\mathrm{d}s \Big)^2 \right]^{\frac{1}{2}}\\
&+ C \mathbb{E} \left(\sup_{0\leq s \leq T}|\tilde K_u(s)|^2 \right) + C \mathbb{E} \left( \int_{0}^{T}  \Big(1+ |Z_s^n|^{2} + \sup_{0\leq y \leq \Upsilon^{(1)}}\varphi^2(y)|Z_s^n|^2 \Big)\mathrm{d}s\right)^2 	
\end{align*}
where, we used H\"older inequality, the elementary inequality: for $\alpha \in (0,1),$ $|z|^{2\alpha} \leq C_{\alpha} + |z|^2$ and $\varphi$ is the smallest continuous function such that $f \leq \varphi$, to obtain the last inequality which is finite from Lemma \ref{lemma 5.1} and Lemma \ref{lem 2.1}. Therefore, Lemma \ref{lemma 1.2.3} yields that $D_ug_n(s,Y_s^n,Z_s^n)$ converges to  $D_ug(s,Y_s,Z_s) $ in the weak topology of  $L^2(\Omega\times [0,T])$. The proof is completed.						
\end{proof}

\subsection{Solvability of the BSDE \eqref{eqmain1} with bounded Malliavin differentiable terminal value }
In this subsection, we provide a necessary condition under which the unique solution $(Y,Z)$ to the BSDE \eqref{eqmain1} belongs to $\mathcal{S}^{\infty}(\mathbb{R})\times \mathcal{S}^{\infty}(\mathbb{R}^d) $. We assume the following:
\begin{itemize}
\item[(M3)] For each $(y,z) \in \mathbb{R}\times \mathbb{R}^d,$ it holds that  $(g(t,y,z))_{t \in [0,T]} \in \mathbb{L}_{1,2p}(\mathbb{R})$ for all $p \geq 1$. Its Malliavin derivative denoted by $(D_ug(t,y,z))_{u,t \in [0,T]}$ satisfies
\begin{align*}
|D_u g(t,y,z)| \leq K_u(t)(1 + |y| + f(|y|) |z|^{\alpha})  \text{ a.s.}
\end{align*}
for any $(u,t,y,z) \in [0,T]\times [0,T]\times \mathbb{R}\times \mathbb{R}^d,$ $\alpha\in (0,1)$ and the process $(K_u(t))_{u,t\in [0,T]}$ is uniformly bounded.
\end{itemize}
\begin{prop}\label{prop 4.6}
Let Assumptions \ref{assum1},\ref{assum2} and (M3) be in force. Assume further that  $\xi$ has a bounded Malliavin derivative i.e. $\|D\xi\|_{\mathcal{S}^{\infty}} = \sup_{0\leq t\leq T}\|D_t\xi\| < \infty \,\,  \mathbb{P}\text{-a.s.}$. Then, the BSDE \eqref{eqmain1} has a unique solution $(Y,Z)\in \mathcal{S}^{\infty}(\mathbb{R})\times \mathcal{S}^{\infty}(\mathbb{R}^d) $
\end{prop}
We will start by proving the following lemma
\begin{lemm}\label{lemm 4.7}
Under the assumptions of Proposition \ref{prop 4.6}, we have
\[  \sup_{n\in \mathbb{N}} \essup_{\omega\in \Omega}\sup_{t \in [0,T]}|Z^n_t(\omega)| < +\infty  \]
where $(Y^n,Z^n)$ stands for the unique solution to the BSDE \eqref{eqmain3}.
\end{lemm}
\begin{proof}
Let us first assume that the drift $g$ is continuously differentiable. Then from Lemma \ref{lemma 5.2}, the solution $\Theta^n =(Y^n,Z^n)_{n\in \mathbb{N}}$ to the BSDE \eqref{eqmain3} belongs to $\mathbb{L}_{1,2} \times (\mathbb{L}_{1,2})^d$ and a version of $\{ (D_u Y^n_t, D_u Z_t^n), 0\leq u,t\leq T  \}$ satisfies for all $t\in [u,T]$ 
\begin{align*}  
D_u Y_t^n =& D_u\xi -\int_{t}^{T} D_u Z_s^n\mathrm{d}B_s\\
& + \int_{t}^{T} \left( (D_ug_n)(s,\Theta_s^n) + \langle (\nabla g_n)(s,\Theta_s^n), D_u\Theta_s^n  \rangle   \right)\mathrm{d}s.  
\end{align*} 
Since, $(D_t Y_t)_{t\in [0,T]}$ is a continuous modification of $(Z_t)_{t\in [0,T]}$, it is sufficient to establish a uniform bound of the process $(D_t Y_t)_{t\in [0,T]}$ to obtain the desired result.
Using once more the assumptions of the Lemma, the properties of the functions $\rho_n, \tilde{\rho}_n$ and $f$ we have
\begin{align*}
	|D_u g_n(t,Y^n_t,Z^n_t)| &\leq K_u(t)(1+ |Y^n_t| + f(|Y^n_t|)|Z^n_t|^{\alpha}) \text{ a.s.}\\
|\nabla_y g_n(t,Y^n_t,Z^n_t)| &\leq \Lambda_y(1+ |Z^n_t|^{\alpha})\text{ a.s.}\\
|\nabla_z g_n(t,Y^n_t,Z^n_t)| &\leq \Lambda_z(1+ f(|Y^n_t|)|Z^n_t|) \text{ a.s.}.
\end{align*}
This implies that: $\sup_{n\geq 1}\|\nabla_z g_n *B \|_{BMO}:= \sup_{n\geq 1}\| \int_0^{\cdot} \nabla_z g_n(s,\Theta_s)\mathrm{d}B_s \|_{BMO} < \infty$ and the measure $\mathbb{P}^n$ with density $\mathrm{d}\mathbb{P}^{n}:= \mathcal{E}(\nabla_z g_n *B)\mathrm{d}\mathbb{P}$ defines an equivalent measure to the probability $\mathbb{P}$. By the Girsanov's theorem the process $(B_t^n)_{t\in[0,T] }$ defined by $B_t^n = B_t -\int_0^t \nabla_z g_n(s,\Theta_s^n)\mathrm{d}s$ is a Brownian motion under the new probability measure $\mathbb{P}^{n}$.
Thus \eqref{MalY^n} can be written is terms of $B^n$ as follows
\begin{equation*}
D_u Y_t^n = D_u\xi -\int_{t}^{T} D_u Z_s^n\mathrm{d}B_s^n
+ \int_{t}^{T} F^n(s,D_uY_s^n,D_uZ_s^n) \mathrm{d}s,
\end{equation*}
where $F^n(s,D_uY_s^n,D_uZ_s^n):= (D_ug_n)(s,\Theta_s^n) + (\nabla_y g_n)(s,\Theta_s^n) D_u\Theta_s^n $. By using once more a standard linearisation technique and applying It\^{o}'s formula to the continuous semimartingale $(e_t^nD_u Y_t^n)_{0\leq t\leq T}$, one obtains
\begin{equation*}
e_t^nD_u Y_t^n = e_T^nD_u\xi -\int_{t}^{T} e_s^nD_u Z_s^n\mathrm{d}B_s^n
+ \int_{t}^{T} e_s^nD_ug_n(s,\Theta_s^n) \mathrm{d}s,
\end{equation*}
where $e_t^n = \exp(\int_0^t (\nabla_y g_n)(s,\Theta_s^n) \mathrm{d}s)$ belongs to $\mathcal{S}^p({\mathbb{R}})$ for all $p >1$. Multiplying both sides of the above equation by $(e_t^n)^{-1}$ and taking the conditional expectation, we have 
\begin{align*}
|D_u Y_t^n| \leq \mathbb{E}^{n}\Big[ (e_t^n)^{-1}e_T^n|D_u\xi| + \int_{t}^{T}(e_t^n)^{-1} e_s^n|D_ug_n(s,\Theta_s^n)| \mathrm{d}s\Big|\mathfrak{F}_t \Big],
\end{align*}
where $\mathbb{E}^n$ stands for the expectation under $\mathbb{P}^n$. Therefore, from the H\"older's inequality, we deduce the existence of a constant $C$ only depending on $\alpha$ and $\sup_{n \in \mathbb{N}}\| Z^n*B \|_{BMO}$ such that 
\begin{align*}
|D_u Y_t^n| \leq C \mathbb{E}^{n}\Big[ |D_u\xi|^{2} + \Big(\int_{t}^{T}|D_ug_n(s,\Theta_s^n)| \mathrm{d}s\Big|\mathfrak{F}_t\Big)^{2} \Big]^{\frac{1}{2}}
\end{align*}
From the assumptions of the Lemma, the first term on the right hand side in the above inequality is uniformly bounded. Let us focus on the second term. Applying successively, Girsanov's theorem, assumption (M3) and H\"older's inequality we obtain:
\begin{align*}
&\sup_{n \in \mathbb{N}}	\mathbb{E}^n\Big( \int_{t}^{T}|D_ug_n(s,\Theta_s^n)| \mathrm{d}s\Big|\mathfrak{F}_t \Big)^{2}\notag\\
=& \sup_{n \in \mathbb{N}} \mathbb{E} \Big(  \mathcal{E}(\nabla_z g_n*B) \int_{t}^{T}|D_ug_n(s,\Theta_s^n)| \mathrm{d}s \Big|\mathfrak{F}_t\Big)^{2}\notag\\
\leq &C  \sup_{n \in \mathbb{N}} \Big[ \mathbb{E}\Big( \mathcal{E}(\nabla_z g_n*B)\Big)^{r}\Big|\mathfrak{F}_t\Big]^{\frac{2}{r}}  \Big[ \mathbb{E}\Big( \int_{t}^{T}|D_ug_n(s,\Theta_s^n)| \mathrm{d}s \Big)^{r'}\Big|\mathfrak{F}_t\Big]^{\frac{2}{r'}}\notag\\
\leq & C  \sup_{n \in \mathbb{N}} \Big[ \mathbb{E}\Big( \int_{t}^{T}(1 + |Y_s^n| + f(|Y_s^n|)|Z_s^n|^{\alpha}) \mathrm{d}s \Big)^{r'}\Big|\mathfrak{F}_t\Big]^{\frac{2}{r'}}\\
\leq &CT +  \sup_{n \in \mathbb{N}} \Big[ \mathbb{E}\Big( \int_{t}^{T}|Z_s^n|^{2} \mathrm{d}s \Big)^{r'}\Big|\mathfrak{F}_t\Big]^{\frac{2}{r'}} \leq CT + Cr'!  \sup_{n\geq 1}\| Z^n*B \|_{BMO}^2 < \infty,
\end{align*}
where the last inequality follows from  energy inequalities and it is finite thanks to Lemma \ref{lemma 5.1}. Thus, 
\[ \sup_{n \in \mathbb{N}} \| Z^n\|_{\mathcal{S}^{\infty}} = \sup_{n \in \mathbb{N}} \| (D_tY^n_t)_{t\in [0,T]}\|_{\mathcal{S}^{\infty}} \leq C. \]
Since, the constant $C$ does not depend on the supremum norms of $\nabla_y g$ and $\nabla_z g$ respectively, then the uniform bound obtained above remains valid when $g$ is not continuously differentiable with respect to its spatial variables by means of standard approximation methods. This concludes the proof.
\end{proof}
\begin{proof}[Proof of Proposition \ref{prop 4.6}:]
From Lemma \ref{lemma 5.1} and Lemma \ref{lemm 4.7}, we can fix $n^{*}\in \mathbb{N}$ such that \[ n^{*} > \sup_{n \in \mathbb{N}}\essup_{\omega\in \Omega} \sup_{t \in [0,T]} \left( |Y_t^n| + |Z_t^n| \right) . \]
We observe $\mathrm{d}t\otimes\mathbb{P}\text{-a.s.}$ that:
\[ g_{n^{*}}\left(t, Y_t^{n^{*}}, Z_t^{n^{*}}\right) = g\left( t, \tilde{\rho}_{n^{*}}(Y_t^{n^{*}}),  {\rho}_{n^{*}}(Z_t^{n^{*}})  \right) = g\left(t, Y_t^{n^{*}}, Z_t^{n^{*}}\right)  \] 
Thus, $\left( Y^{n^{*}}, Z^{n^{*}} \right)$ solves the BSDE \eqref{eqmain1} in $\mathcal{S}^{\infty} \times \mathcal{S}^{\infty}$. This ends the proof.
\end{proof}

\subsection{Classical differentiability}\label{prelm}
Throughout this section, we consider the following parametrised BSDE 
\begin{equation}\label{bsde}
Y_t^x = \xi(x) + \int_{t}^{T} g(s,\omega,x, Y_s^x,Z_s^x)\mathrm{d}s - \int_{t}^{T} Z_s^x\mathrm{d}B_s,\quad t\in [0,T], x\in \mathbb{R}^m.
\end{equation}
We suppose the following assumption:
\begin{assum}\label{assum4}\leavevmode
\begin{itemize}
\item[(C1)] Let $m,d \in \mathbb{N}$. Let $g: [0,T]\times\Omega\times \mathbb{R}^m\times \mathbb{R}\times \mathbb{R}^d \rightarrow \mathbb{R}$ be an adapted measurable function, differentiable in the spatial variables with continuous partial derivatives in $y$ and $z$. There exist a positive process $(K_t(x))_{t\in [0,T]}$ depending on $x \in \mathbb{R}^m$ and  a locally bounded and non-decreasing function $f\in L^1_{loc}(\mathbb{R},\mathbb{R}_{+})$ such that for all $(t,x,y,z)\in [0,T]\times \mathbb{R}^m\times \mathbb{R}\times \mathbb{R}^d,$ $\alpha \in (0,1)$
\begin{align*}
	|g(t,x,y,z)|&\leq \Lambda_0 + \Lambda_y|y|+ \Lambda_z( |z| + f(|y|)|z|^2) \text{ a.s.,}\\
	|\nabla_x g(t,x,y,z)| &\leq K_t(x)(1 + |y|+ f(|y|)|z|^\alpha)\text{ a.s.,}\\
	|\nabla_y g(t,x,y,z)| &\leq \Lambda_y(1 +|z|^{\alpha}) \text{ a.s.,}\\
	|\nabla_z g(t,x,y,z)| &\leq \Lambda_z(1+ f(|y|)|z|)\text{ a.s.}
\end{align*}
Furthermore, the process $(K_t(x))_{t\in [0,T]}$ satisfies $\sup_{x\in \mathbb{R}^m } \int_{0}^{T} \mathbb{E}|K_s(x)|^{2p} \mathrm{d}s < \infty$ for any $p\geq 1.$
\item[(C2)] 
\begin{itemize}
	\item[(i)]
	For any $x\in \mathbb{R}^m$, the random variable $\xi(x)$ is $\mathfrak{F}_T\text{-adapted}$ and $\sup_{x\in \mathbb{R}^m}\| \xi(x)\|_{L^{\infty}(\Omega)} < \infty$  a.s.
	\item[(ii)] For all $p\geq 1$ the mapping $x\mapsto \xi(x)$ from $\mathbb{R}^m$ to $L^{2p}(\Omega)$ is differentiable and  $\sup_{x\in \mathbb{R}^m}\| \nabla_x\xi(x)\|_{L^{2p}(\Omega)} < \infty.$
	
\end{itemize}
\item[(C3)] The function $x\mapsto \nabla_x\xi(x)$ is continuous.
\end{itemize}
\end{assum}
Note that even if the process $(K_t(x))_{t\in [0,T]}$  does not satisfy one of the standard requirements in the literature (that is $\mathbb{E}\sup_{0\leq t \leq T}|K_t(x)|^p < \infty$), we can still use the H\"{o}lder  and the Minkowski's integral type inequalities to prove the desire result.

\begin{lemm}\label{lemma 4.3}
Suppose Assumption \ref{assum4} is valid. For all $p > 1$ and $i \in \{ 1,\ldots,m\}$ there exists $C > 0$ such that for all $x,x' \in \mathbb{R}^m$ and $h,h' \in \mathbb{R}$ for which $(x+he_i)$ and $(x'+h'e_i)$ belongs to $\mathbb{R}^m$ we have
\begin{align}\label{estim1}
&\mathbb{E} \Big[ \sup_{t\in [0,T]} |Y_t^{x+he_i} -Y_t^{x'+h'e_i}|^{2p} + \Big(  \int_{0}^{T}\!\! |Z_s^{x+he_i} -Z_s^{x'+h'e_i}|^2 \mathrm{d}s \Big)^p   \Big]\\
&\leq C (|x-x'|^2 + |h-h'|^2)^p,\notag
\end{align}
where $(Y^r,Z^r) $ is the solution to the BSDE \eqref{bsde} with parameter $r \in \{ x+ he_i, x'+h'e_i\}.$
\end{lemm}
\begin{remark} 
It follows rom the Kolmogorov's continuity criterion,for $0\leq t \leq T$ the mapping $x\mapsto Y^x_t$ has a continuous version for which almost all sample paths are $\beta\text{-H\"{o}lder}$ continuous in $\mathbb{R}^m$ for any $\beta \in (0,1).$ For $(t,x)\in [0,T]\times\mathbb{R}^m$ the mapping $t\mapsto Y_t^x(\omega)$ is continuous $\mathbb{P}\text{-a.s. } \omega \in \Omega$. This is a necessary condition to obtain a classical differentiability result for the solution process $(Y^x,Z^x)$ to the BSDE \eqref{bsde} under Assumption \ref{assum4}.
\end{remark}
\begin{proof}\leavevmode
Set $\bar{x}=x-x', \bar{h}=h-h'$ and define the following processes: $\delta Y_t := (Y_t^{x+he_i}- Y_t^{x'+h'e_i})$, $\delta Z_t := (Z_t^{x+he_i}- Z_t^{x'+h'e_i})$ and $\delta \xi := (\xi(x+he_i)- \xi(x'+h'e_i))$. From assumption (C1) $(\delta Y, \delta Z)$ satisfies the equation
\begin{align*}
\delta Y_t = \delta \xi - \int_{t}^{T}\left[ I_s^y(\bar{x} + \bar{h}e_i) + I_s^x\delta Y_s + I_s^z \delta Z_s \right]\mathrm{d}s -\int_{t}^{T} \delta Z_s \mathrm{d}B_s,
\end{align*}
where the processes $I^x,I^y$ and $I^z$ are given by
\begin{align*}
I_s^x &= \int_{0}^{1} (\nabla_x g)(s,x+he_i+\theta(\bar{x} + \bar{h} e_i),Y_s^{x+he_i},Z_s^{x+he_i})\mathrm{d}\theta,\\
I_s^y &= \int_{0}^{1} (\nabla_y g)(s,x'+h'e_i ,Y_s^{x+he_i}-\theta \delta Y_s,Z_s^{x+he_i})\mathrm{d}\theta,\\
I_s^z &= \int_{0}^{1} (\nabla_z g)(s,x'+h'e_i ,Y_s^{x+he_i},Z_s^{x+he_i}-\theta \delta Z_s)\mathrm{d}\theta.
\end{align*}
Using (C1), the following bounds can be obtained:
\begin{align*}
|I_t^x| &\leq \int_{0}^{1} \left[ (1 + |Y_t^{x+he_i}| + [f(|Y_t^{x+he_i}|)|Z_t^{x+he_i}|]^{\alpha}) K_t(x+he_i +\theta(\bar{x}+\bar{h}e_i)) \right]\mathrm{d}\theta,\\
|I_t^y| &\leq \Lambda_y  (1 + |Z_t^{x+he_i}|^{\alpha} ),\\
|I_t^z|&\leq \Lambda_z \big( 1+ f(|Y_t^{x+he_i}|)(|Z_t^{x+he_i}| + |Z_t^{x'+h'e_i}|)\big).
\end{align*}
Observe that $(I^z*B)$ is a BMO martingale and 
\[ \sup_{\tiny x\in \mathbb{R}^m, h\in \mathbb{R}} (\|I^z*B\|_{BMO})\leq \sup_{\tiny x\in \mathbb{R}^m, h\in \mathbb{R}} (\|Z^{x+he_i}*B\|_{BMO} + \|Z^{x}*B\|_{BMO}) <\infty. \] 
Moreover, from (C2), $\delta \xi$ is bounded. Thus,
Lemma \ref{thap} (see also \cite[Lemma A 1]{ImkRevRich}) yields that for any $p> 1$:
\begin{align}\label{est1}
\|\delta Y\|_{\mathcal{S}^{2p}}^{2p}  + \|\delta Z\|_{\mathcal{H}^{2p}}^{2p} \leq C\mathbb{E} \Big[ |\delta \xi|^{2pq} + \Big( \int_{0}^{T} |I_s^x|(|x-x'|+|h-h'|) \mathrm{d}s\Big)^{2pq}  \Big]^{\frac{1}{q}} ,
\end{align}
where $q\in (1,\infty).$ 
From (C2) and the mean value theorem, we deduce the existence of a constant $C>0$ such that
\[ \mathbb{E}|\delta \xi|^{2pq} \leq C (|x-x'|^2 + |h-h'|^2)^{pq} ,\] 
We now focus on the second term of  \eqref{est1}. For simplicity, we will write $K_t^{\bar{x},\bar{h}}(\theta):= K_t(x+he_i +\theta(\bar{x}+\bar{h}e_i))$. For any $\gamma \geq 2$, using the H\"{o}lder inequality and Minkowski's integral type inequality we have 
\begin{align*}
&\Big( \int_{0}^{T} |I_t^x|\mathrm{d}t \Big)^{\gamma}\\
\leq & \Big( \int_{0}^{T} (1 + |Y_t^{x+he_i}| + f(|Y_t^{x+he_i}|)|Z_t^{x+he_i}|^{\alpha})  \int_{0}^{1}  K_t^{\bar{x},\bar{h}}(\theta)\mathrm{d}\theta \mathrm{d}t \Big)^{\gamma}\\
\leq &\Big(\int_{0}^{T} \big(1 + |Y_t^{x+he_i}| + f(|Y_t^{x+he_i}|)|Z_t^{x+he_i}|^{\alpha}\big)^{\frac{1}{\alpha}}\mathrm{d}t \Big)^{\alpha\gamma}  \Big( \int_{0}^{T}\Big(\int_{0}^{1}  K_t^{\bar{x},\bar{h}}(\theta) \mathrm{d}\theta\Big)^{\frac{1}{1-\alpha}} \mathrm{d}t \Big)^{\gamma(1-\alpha)}\\
\leq &C \Big(\int_{0}^{T} \big(1 + |Y_t^{x+he_i}|^{\frac{1}{\alpha}} +\sup_{0\leq y \leq 2\Upsilon^{(1)}}\varphi^{\frac{1}{\alpha}}(y)|Z_t^{x+he_i}|^2\big)\mathrm{d}t \Big)^{\alpha\gamma} \\
& \times \Big( \int_{0}^{1}\Big(\int_{0}^{T}  |K_t^{\bar{x},\bar{h}}(\theta)|^{\frac{1}{1-\alpha}} \mathrm{d}t\Big)^{(1-\alpha)} \mathrm{d}\theta \Big)^{\gamma}.
\end{align*}
Taking the expectation on both sides of the above inequality and applying the H\"{o}lder's inequality, we have
\begin{align*}
&\mathbb{E}\Big( \int_{0}^{T} |I_t^x|\mathrm{d}t \Big)^{\gamma} \\
\leq &C \Big[ \mathbb{E}\Big(\int_{0}^{T} \big(1 + |Y_t^{x+he_i}|^{\frac{1}{\alpha}} + \sup_{0\leq y \leq 2\Upsilon^{(1)}}\varphi^{\frac{1}{\alpha}}(y)|Z_t^{x+he_i}|^2\big)\mathrm{d}t\Big)^{\gamma} \Big]^{\alpha}\\
&\times \Big[ \mathbb{E}\Big( \int_{0}^{1}\Big(\int_{0}^{T}  |K_t^{\bar{x},\bar{h}}(\theta)|^{\frac{1}{1-\alpha}} \mathrm{d}t\Big)^{(1-\alpha)} \mathrm{d}\theta \Big)^{\frac{\gamma}{1-\alpha}}\Big]^{(1-\alpha)} \\
\leq &C \Big[ \mathbb{E}\Big(\int_{0}^{T} \big(1 + |Y_t^{x+he_i}|^{\frac{1}{\alpha}} + \sup_{0\leq y \leq 2\Upsilon^{(1)}}\varphi^{\frac{1}{\alpha}}(y)|Z_t^{x+he_i}|^2\big)\mathrm{d}t\Big)^{\gamma} \Big]^{\alpha} \\
&\times\Big[ \mathbb{E} \int_{0}^{1}\Big(\int_{0}^{T}  |K_t^{\bar{x},\bar{h}}(\theta)|^{\frac{1}{1-\alpha}} \mathrm{d}t\Big)^{(1-\alpha)\cdot{\frac{\gamma}{1-\alpha}}}\mathrm{d}\theta\Big]^{(1-\alpha)} \\
\leq &C T^{\frac{1}{(1-\gamma)}}\Big[ \mathbb{E}\Big(\int_{0}^{T} \big(1 + |Y_t^{x+he_i}|^{\frac{1}{\alpha}} + \sup_{0\leq y \leq 2\Upsilon^{(1)}}\varphi^{\frac{1}{\alpha}}(y)|Z_t^{x+he_i}|^2\big)\mathrm{d}t\Big)^{\gamma} \Big]^{\alpha}\\
&\times \Big[ \int_{0}^{1}\int_{0}^{T}  \mathbb{E}|K_t^{\bar{x},\bar{h}}(\theta)|^{\frac{\gamma}{1-\alpha}} \mathrm{d}t \mathrm{d}\theta \Big]^{(1-\alpha)}.
\end{align*}
Hence, we deduce that 
\begin{align*}
&\sup_{\tiny x\in \mathbb{R}^m, h\in \mathbb{R}}\mathbb{E}\Big( \int_{0}^{T} |I_t^x|\mathrm{d}t \Big)^{\gamma}\\
\leq & C \sup_{\tiny x\in \mathbb{R}^m, h\in \mathbb{R}} \Big(1 + \|Y^{x+he_i} \|_{\mathcal{S}^{\infty}}^{\frac{\gamma}{\alpha}} +  \|\sqrt{\sup_{0\leq y \leq 2\Upsilon^{(1)}}\varphi^{\frac{1}{2\alpha}}(y)}Z^{x+he_i} \|_{\mathcal{H}^{\gamma}}^{2\gamma\alpha}  \Big)\\
&\times \sup_{\tiny r\in \mathbb{R}^m} \Big(\int_{0}^{T}  \mathbb{E}|K_t(r)|^{\frac{\gamma}{1-\alpha}} \mathrm{d}t\Big)^{(1-\alpha)}<\infty.
\end{align*}
The proof is completed.
\end{proof}

It is then possible to identify the couple process $(\nabla Y^x, \nabla Z^x )$ as the derivatives of $(Y^x,Z^x)$ under Assumption \ref{assum4} such that the following BSDE 
\begin{align}\label{bsde1}
\nabla Y_t^x &= \nabla_x \xi(x) - \int_{t}^{T}\!\! \nabla_x Z_s^x\mathrm{d}B_s\\
& + \int_{t}^{T}\!\!  \Big( \nabla_x g(s,\Theta_s^x) + \nabla_y g(s,\Theta_s^x) \nabla Y^x_s + \nabla_z g(s,\Theta_s^x) \nabla Z^x_s\Big) \mathrm{d}s ,\notag
\end{align}
makes sense for all $x\in \mathbb{R}^m$, $t\in [0,T],$ where $\Theta_{\cdot}^x = (x,Y_{\cdot}^x, Z_{\cdot}^x)$.

We only prove the differentiability of $(Y^x,Z^x)$ with respect to the natural topological structure on the Banach space $\mathcal{S}^{2p}(\mathbb{R})\times \mathcal{H}^{2p}(\mathbb{R}^d)$, $p> 1$. Under some additional assumptions (for example twice differentiability of the parameters) the pathwise differentiablity of the map $x\mapsto (Y_t^x(\omega),Z_t^x(\omega))$ for almost all $(\omega,t) \in \Omega\times[0,T]$ can be obtained without major difficulties.

The main result of this subsection is the following 
\begin{thm}[Differentiability]\label{diffY} Suppose the coefficients of the BSDE \eqref{bsde} satisfy Assumption \ref{assum4}. Then for any parameter $x\in \mathbb{R}^n$ and all $p> 1$ the solution function: $\mathbb{R}^m\rightarrow \mathcal{S}^{2p} \times \mathcal{H}^{2p}$,  $x\mapsto (Y^x,Z^x)$ is differentiable in the norm topology and the couple (derivatives) $x\mapsto (\nabla Y^x,\nabla Z^x)$ is the unique solution to the BSDE \eqref{bsde1}.
In particular for $x,x' \in \mathbb{R}^m$ we have 
\[  
\lim_{x\rightarrow x'} \left\{ \| \nabla_x Y^{x}_t - \nabla_x Y^{x'}_t  \|_{\mathcal{S}^{2p}}^{2p} + \| \nabla_x Z^{x}_t - \nabla_x Z^{x'}_t  \|_{\mathcal{H}^{2p}}^{2p}  \right\} = 0.  
\]	
\end{thm}
The proof of Theorem \ref{diffY} follows from {well known} techniques related to differentiability of BSDEs with drivers that grow quadratically in the control variable. We refer the reader for example to \cite{ArImDR, FreiDosReis, ImkDos, DosReis}. For the sake of better understanding, we sketch the proof below

\begin{proof}[Proof of Theorem \ref{diffY}] 
We start by observing the following: Assumption (H1) guarantees the existence of a maximal solution $(Y^x,Z^x)$ to the BSDE \eqref{bsde} in  $\mathcal{S}^{\infty}\times\mathcal{H}_{BMO}$ such that the norms of $Y^x$ (resp. $Z^x$) in $\mathcal{S}^{\infty}$ (resp. $\mathcal{H}_{BMO}$) are uniformly bounded in $x$. Using Lemma \ref{lemma 4.3} we also deduce that for all $p >1$, $x\in \mathbb{R}^m$, $h,h'\in \mathbb{R}$ and $i\in \{ 1,\cdots,m\}$
\[  \lim_{h\rightarrow 0}\left\{  \Big\Vert {Y_t^{x+he_i}-Y_t^x} \Big\Vert_{\mathcal{S}^{2p}} + \Big\Vert {Z_t^{x+he_i}-Z_t^x}\Big\Vert_{\mathcal{H}^{2p}}  \right\} = 0 , \]	
Then, for any candidate $(\nabla_{x_i}Y^x_t,\nabla_{x_i}Z^x_t)$ for the partial derivatives satisfying BSDE \eqref{bsde1}, the following can be achieved as in \cite[P33]{DosReis}:  
\[  \lim_{h\rightarrow 0}\Big\{  \Big\Vert \frac{Y_t^{x+he_i}-Y_t^x}{h}- \nabla_{x_i}Y^x_t \Big\Vert_{\mathcal{S}^{2p}} + \Big\Vert \frac{Z_t^{x+he_i}-Z_t^x}{h}- \nabla_{x_i}Z^x_t \Big\Vert_{\mathcal{H}^{2p}}  \Big\} = 0 .\]	
The above guarantees the existence of partial derivatives of $(Y^x,Z^x)$. Using Lemma \ref{lemma 4.3} once more, and the continuity of the derivatives of $g$, it can be shown that $(\nabla_{x_i} Y^x, \nabla_{x_i} Z^x)$ are continuous with respect to any $x\in \mathbb{R}^m$. This concludes the proof.
\end{proof}

\section{Differentiability of quadratic FBSDEs with rough drift}\label{Appli}

In this section, we study the smoothness properties of the solution $\boldsymbol{X}^x= (X^x,Y^x,Z^x)$ of the following FBSDE
\begin{align}
X_t^x =& x + \int_{0}^{t} b(s,X_s^x)\mathrm{d}s + B_t, \label{sde}\\
Y_t^x = &\phi(X_T^x) + \int_{t}^{T} g(s,X_s^x,Y_s^x,Z_s^x)\mathrm{d}s - \int_{t}^{T} Z_s^x\mathrm{d}B_s,\label{bsde3}
\end{align}
where the functions $b, \phi$ and $g$ are both deterministic explicit functional that are Borel measurables and satisfying some conditions that will be made precise below. 

In the sequel, the driver $g$ and the terminal value $\phi$ satisfy the following assumptions:
\begin{itemize}
\item[(AY):] The function $\phi: \mathbb{R}^d\rightarrow \mathbb{R}$ is continuous, measurable and uniformly bounded; $g: [0,T]\times \mathbb{R}^d\times \mathbb{R}\times \mathbb{R}^d \rightarrow \mathbb{R}  $ is a measurable function satisfying: $\|g(t,0,0,0)\|_{\infty} \leq \Lambda_0 $ and there exist non negative constants $\Lambda_x,\Lambda_y\Lambda_z$ and $\Lambda_{\phi}$ such that for all $(t,x,y,z)\in [0,T]\times \mathbb{R}^d\times \mathbb{R}\times \mathbb{R}^d,$ $(t,x',y',z')\in [0,T]\times \mathbb{R}^d\times \mathbb{R}\times \mathbb{R}^d,$ $\alpha \in (0,1)$
\begin{align*}
|g(t,x,y,z)-g(t,x',y,z)| &\leq \Lambda_x (1 + |y|+f(|y|)|z|^\alpha)|x-x'|, \\
|g(t,x,y,z)-g(t,x,y',z')| &\leq \Lambda_y (1 +(|z|^{\alpha}+|z'|^{\alpha}))|y-y'| \\
& + \Lambda_z(1+ (f(|y|)+f(|y'|))(|z|+|z'|)) |z-z'|,\\
| \phi(x)-\phi(x')|&\leq \Lambda_{\phi}|x-x'|
\end{align*}
where $f\in L^1_{loc}(\mathbb{R},\mathbb{R}_{+})$ is locally bounded and non-decreasing.
\item[(AY1):] The functions $\phi$ and $g$ are continuously differentiable in $(x,y,z)$ 
\end{itemize}
\subsection{The case of SDEs with bounded drift}
In this section, we assume that the drift $b$ in the forward SDE \eqref{sde} satisfies
\begin{itemize}
\item[(AX):] $b: [0,T]\times \mathbb{R}^d\rightarrow \mathbb{R}^d$ is uniformly bounded and Borel measurable.
\end{itemize}

It is well known that under (AX), the SDE \eqref{sde} has a unique strong Malliavin and Sobolev differentiable solution (see for example \cite{ MMNPZ13, NilssenProske}). Let us recall the following results whose proofs can be found in \cite{ MMNPZ13, NilssenProske}.
\begin{thm}\label{thmpr1}Suppose (AX) is valid. 
\begin{itemize}
\item[(i)]	The SDE \eqref{sde} has a unique strong solution $ X^x_t\in L^2(\Omega;W_{loc}^{1,2}(\mathbb{R}^d))$. Moreover, for all $s,t \in [0,T]$, $x,y \in \mathbb{R}$ it holds
\begin{align}\label{eq 6.3}
	\mathbb{E}\left[|X_t^x - X_s^y|^2\right] \leq C(\|b\|_{\infty})\left( |t-s| + |x-y|^2  \right).
\end{align}
\item[(ii)] The strong solution $X^x$ is Malliavin differentiable and for $0\leq s\leq t \leq T$, the Malliavin derivative $D_s X_t^x$ satisfies
\begin{align}
	\sup_{x\in \mathbb{R}}\sup_{s\in [0,t]} \mathbb{E}\left[ | D_s X_t^x |^2  \right] &\leq C_{d}(\|b\|_\infty),\label{eq 6.80}
\end{align}
 where $C$ is an increasing function of $\|b\|_\infty$ depending only on $d$. Moreover, the following relation holds
\begin{equation}\label{eq 6.7}
	D_sX_t^x \left( \nabla_x X_s^x\right) = \nabla_x X_t^x \quad \mathrm{d}t\otimes\mathbb{P}\text{-a.s.},
\end{equation}
where $\nabla_x X_t^x $ denotes the first variation  of the process $x\mapsto X^x$.
\end{itemize}
\end{thm}

Combining the results in \cite{ MMNPZ13, NilssenProske,MiNuaSan89}, we obtain the following
\begin{lemm}
The strong solution $X^x$ is Malliavin differentiable and for $0\leq s\leq t \leq T$, the Malliavin derivative $D_s X_t^x$ satisfies
\begin{align}
	\sup_{x\in \mathbb{R}}\sup_{s\in [0,t]} \mathbb{E}\left[ | D_s X_t^x |^p  \right] &\leq C_{d,p}(\|b\|_\infty)\label{eq 6.8}
\end{align}
for all $p \geq 1$, where $C$ is an increasing function of $\|b\|_\infty$ depending only on $d$ and $p$. 
\end{lemm}
\begin{proof}
Since $b$ is bounded, the solution to the SDE \eqref{sde} belong to $L^p$ for all $p \geq 1$. In addition, there exists an approximating sequence of Malliavin differentiable solutions  $(X^{n,x})_{n\geq 1}$  that converges in $L^2$ to  $X^{x}$. It follows that $(X^{n,x})_{n\geq 1}$  converges to  $X^{x}$ in $L^p$ for all $p \geq 1$.
In addition, one can show as in \cite[Proposition 7]{NilssenProske} that 
$$
\sup_{x\in \mathbb{R}}\sup_{s\in [0,t]} \sup_{n}\mathbb{E}\left[ | D_s X_t^{n,x} |^p  \right] \leq C_{d,p}(\|b\|_\infty).
$$
The result then follows by a compactness criterion argument. 
	\end{proof}
The next result pertains with  the existence and representation of the Malliavin derivative of the solution to the FBSDE \eqref{sde}-\eqref{bsde3}. 

\begin{thm}\label{Main}
Suppose Assumptions (AX), (AY) and (AY1) are in force. Then the solution process $(Y^x,Z^x)$ belongs to $\mathbb{L}_{1,2} \times (\mathbb{L}_{1,2})^d$ and a version of $(D_uY_t^x,D_uZ_t^x)_{u,t\in [0,T]}$ is the unique solution to the linear BSDE 
\begin{align}\label{5.7}
D_u Y_t^x &= 0 \text{ and } D_u Z_t^x = 0, \text{ if } t\in [0,u),\nonumber\\
D_u Y_t^x & = \nabla_x \phi(X_T^x)D_u X_T^x + \int_{t}^{T} \langle \nabla g(s,\boldsymbol{X}^x),D_u \boldsymbol{X}_s^x \rangle \mathrm{d}s  - \int_{t}^{T} D_u Z_s^x \mathrm{d}B_s, \quad t\in [u,T],
\end{align}
where $D_uX^x$ is the Malliavin derivative of the process $X^x$.

\end{thm}
\begin{proof}
We just need to prove as in \cite{ArImDR, ImkDos, ImkDosReis, DosReis} that the conditions of the theorem imply that assumptions (M1) and (M2) in Section \ref{Maldif} are satisfied. 

Here $\xi(x)= \phi(X_T^x)$ and for $u\in [0,T]$ we have
\[|D_u\xi(x)|= |D_u\phi(X_T^x)|=|(D_uX_T^x)\nabla_x\phi(X_T^x)| \leq C |D_uX_T^x|.\]
Using \eqref{eq 6.8}, it holds that for any $p\geq 1$ 
\begin{align*}
\sup_{x\in \mathbb{R}}\sup_{u\in [0,T]}\mathbb{E}\left[ |D_u\xi(x)|^{p}\right]\leq C \sup_{x\in \mathbb{R}}\sup_{u\in [0,T]}\mathbb{E} \left[ |D_uX_T^x|^{p}\right] < \infty.
\end{align*}
On the other hand, consider 
\begin{align*}
\bar{g}: \Omega\times [0,T]\times \mathbb{R} &\rightarrow \mathbb{R}\\
(\omega,t,x,y,z)&\mapsto g(t,X_t^x(\omega),y,z).
\end{align*}
From (AY1), it is readily seen that $\bar{g}$ satisfies (M1) almost surely. Furthermore, 
\[  |D_u \bar{g}(t,x,y,z)|\leq \Lambda_x|D_uX_t^x|(1 +|y| + f(|y|)|z|^{\alpha}) .\]
Using once more \eqref{eq 6.8} one has  $\sup_{0\leq t \leq T}\int_{0}^{T} \mathbb{E}[|D_u X_t^x|^{2p}]\mathrm{d}u < \infty$ for any $p\geq 1$. Thus (M2) is satisfied.

\end{proof}

\begin{remark}
In the one dimensional case, under (AX), the first variation process and the Malliavin derivative of the solution to the SDE \eqref{sde} can be represented explicitly (see for example \cite{BMBPD17, MenTan19}) with respect to the local tome space-time integral $\mathrm{d}t\otimes \mathrm{d}\mathbb{P}\text{-a.s.}:$  
\begin{align*}
\nabla_x X_t^x &= \exp\Big( -\int_{0}^{t}\int_{\mathbb{R}} b(u,y)L^{X^x}(\mathrm{d}u,\mathrm{d}y)  \Big) ,\\
D_s X_t^x &= \exp\Big( -\int_{s}^{t}\int_{\mathbb{R}} b(u,y)L^{X^x}(\mathrm{d}u,\mathrm{d}y)  \Big).
\end{align*}	
\end{remark}
%
%

\subsection{The case of SDEs with $C_b^{\beta}$ drift } In this section, we assume that the drift $b$ of the forward SDE satisfies:
\begin{assum}\label{asumb2}\leavevmode
There exists $\beta \in (0,1)$ such that $b\in L^{\infty}([0,T];C_b^{\beta}(\mathbb{R}^d;\mathbb{R}^d))$.
\end{assum}

The following result is from \cite{FlanGubiPrio10}

\begin{thm}\label{theoGFP} Suppose Assumption \ref{asumb2} and fix any $\beta' \in (0,\beta)$. Then 
\begin{enumerate}
\item (Pathwise uniqueness) For every $s \geq 0, x\in \mathbb{R}^d$ the SDE \eqref{sde} has a unique continuous adapted solution $ X^{s,x}=(X^{s,x}_t(\omega),t\geq s,\omega\in \Omega)$.
\item (Differentiable flow) There exists a stochastic flow $X^{s,x}_t=\phi_{s,t}(x)$ of diffeomorphisms for equation \eqref{sde}. The flow is also of class $C^{1+\beta'}$ and for any $p\geq 1$
\[ \sup_{x\in \mathbb{R}} \sup_{0\leq s \leq T} \mathbb{E}\Big[ \sup_{s\leq t \leq T}\|D\phi_{s,t}(x)\|^p  \Big] < \infty .\]
\end{enumerate}
\end{thm}
Similar to \cite{FlanGubiPrio10}, we have
\begin{prop}\label{cor 5.7}
Under Assumption \ref{asumb2}, the solution $(X_t^x,0\leq t\leq T)$ to equation \eqref{sde} is Malliavin differentiable and for any $p\geq 2$
\begin{equation}\label{5.12}
\sup_{0\leq s\leq t} \mathbb{E}\big[ \sup_{s\leq t \leq T}|D_sX_t^x|^p  \big] < \infty .
\end{equation}
\end{prop}
\begin{proof} Fix $\lambda > 0$ and consider the following backward Kolmogorov PDE 
\[ \partial_t u_{\lambda} + \mathcal{L}u_{\lambda}- \lambda u_{\lambda} = -b, \quad (t,x) \in [0,+\infty)\times \mathbb{R}^d, \]
where $\mathcal{L}u_{\lambda} = 1/2\Delta u_{\lambda} + b\cdot Du_{\lambda}.$ It can be shown (see \cite[Lemma 6]{FlanGubiPrio10}) that $u_{\lambda} \in L^{\infty}([0,\infty); C_b^{2+\beta}(\mathbb{R}^d))$. In addition for $\lambda$ large enough, the map $\Psi_{\lambda}(t,x) = x + u_{\lambda}(t,x) $ satisfies (see \cite[Lemma 6]{FlanGubiPrio10}) :
\begin{enumerate}
\item[(i)] uniformly in $t$, $\Psi_{\lambda}$ has bounded first and second derivatives and the second (Fr\'{e}chet) derivative $D^2\Psi_{\lambda}$ is globally $\beta\text{-H\"{o}lder}$ continuous.
\item[(ii)] For any $t\geq 0,$ $\Psi_{\lambda} : \mathbb{R}^d\mapsto \mathbb{R}^d$ is non-singular diffeomorphism of class $C^2$.
\item[(iii)] $\Psi_{\lambda}^{-1}$ has bounded first and second derivatives uniformly in $t \in [0,\infty)$.
\end{enumerate}
Let us consider the following SDE 
\begin{align}\label{auxi}
\tilde{X}_t = y + \int_{0}^{t} \tilde{b}(v, \tilde{X}_v)\mathrm{d}v + \int_{0}^{t} \tilde{\sigma}(v, \tilde{X}_v)\mathrm{d}B_v ,\quad t \in [0,T],
\end{align}
where $\tilde{b}(t,y) = -\lambda u_{\lambda}(t,\Psi_{\lambda}^{-1}(t,y))$ and $\tilde{\sigma}(t,y) = D \Psi_{\lambda} (t,\Psi_{\lambda}^{-1}(t,y))$. It is then clear that: $\tilde{b} \in L^{\infty}([0,\infty); C_b^{2+\beta}(\mathbb{R}^d))$ and $\tilde{\sigma} \in L^{\infty}([0,\infty); C_b^{1+\beta}(\mathbb{R}^d))$. Thus, equation \eqref{auxi} has a unique strong Malliavin differentiable solution (see \cite[Theorem 2.2.1]{Nua06}) such that for any $p\geq 2$ 
\begin{align*}
\sup_{0\leq s\leq t} \mathbb{E}[ \sup_{s\leq t \leq T}|D_s\tilde X_t|^p ] < \infty.
\end{align*}
Let $(\tilde{X}_t,0\leq t\leq T)$ be the solution to the SDE \eqref{auxi}. Then we deduce that $X_t = \Psi_{\lambda}^{-1}(t,\tilde X_t)$ is solution to SDE \eqref{sde}. Using the chain rule for Malliavin calculus  and the fact that $\Psi_{\lambda}^{-1}$ has bounded first derivative, we have
\begin{align*}
\mathbb{E}[ \sup_{s\leq t \leq T}|D_s X_t|^p ]\leq C \mathbb{E}[ \sup_{s\leq t \leq T}|D_s\tilde X_t|^p ]< \infty. 
\end{align*}
Thus, \eqref{5.12} follows. 
\end{proof}
\begin{remark}\label{kun}
By using the same argument as above, one can show that :
$$\mathbb{E}\Big[ \sup_{0\leq t \leq T}|\nabla_x X_{t}^x|^{p} \Big]<\infty,\,\, \forall p \in ]-\infty, 0[ \cup ]0,+\infty[\footnote{The case $p \geq 1$ was treated in Theorem \ref{theoGFP} ( see \cite{FlanGubiPrio10})}.$$
Indeed, the solution $\tilde{X}$ to the equation \eqref{auxi} \textup{(}with Lipschitz continuous coefficients $\tilde{b}$ and $\tilde{\sigma}$\textup{)} is differentiable with respect to x and  $(\nabla_x  \tilde X_{t})^{-1}$ satisfies a linear SDE such that $\mathbb{E}\left[ \sup_{0\leq t \leq T}|(\nabla_x \tilde X_{t})^{-1}|^{p} \right]<\infty$, for any $p>1$ \textup{(}see \cite{Kun90}\textup{)}. Then, by using the fact that $\Psi^{-1}$ is differentiable with bounded derivative, we obtain the desired result.
\end{remark}

Note that if the drift $b$ is such that $\dive b \in L^1([0,T];L^1_{\text{loc}}(\mathbb{R}^d))$, then using Liouville's theorem, one can show that the Malliavin derivative of the forward process satisfies the following Euler identity:
\begin{equation}
\det(D_s X_t^x) = 1_{\{ 0\leq s\leq t\}}\exp\left( \int_s^t \dive b(v,X_v^x) \mathrm{d}v\right),
\end{equation}
where $\det(\cdot)$ denotes the determinant of a matrix. In this work, $b$ is not differentiable, then the above representation is not valid. 
Assuming that $b$ is H\"older continuous, a representation of the first variation process $\nabla_x X^x_t$ of the SDE \eqref{sde} was derived in terms of a system of Young type equations in \cite{Khoa}. More precisely

\begin{thm}[Theorem 4.1 in \cite{Khoa}]\label{Khoa}
Suppose Assumption \ref{asumb2}. For any $\beta' \in (0,\beta)$, the first variation process $\nabla X_t^x$ satisfies the following system of Young-type equations
\begin{align}\label{young1}
	\partial_{x_i}X_t^{j,x} = \delta_{i,j} + \sum_{k=1}^{d}\int_{0}^{t}\partial_{x_i}X_s^{j,x}\mathrm{d}\boldmath{V}_s^{k,j}(b,X), \quad \forall i,j\in \{1,\cdots,d\}
\end{align}
where $\delta_{i,j}$ stands for the Kronecker delta symbol, $\boldsymbol{V}_t^{k,j}(b,X)= \mathcal{A}_t^X[\partial_k b^j]$ and $(\mathcal{A}_t)_{s\leq t \leq T}$ is the unique (up to modifications) stochastic process with values in $\mathbb{R}^d$ satisfying conditions of Theorem 2.3 in \cite{Khoa}).
Moreover, the map $t\rightarrow \nabla X_t^x$ is a.s. $\frac{1+\beta'}{2}\text{-H\"{o}lder}$ continuous for every $\beta' \in (0,\beta)$.
\end{thm}
The main ingredient in the proof of the above result is to rigorously establish the well posedness of the process $\boldsymbol{V}_t^{k,j}(b,X)= \mathcal{A}_t^X[\partial_k b^j]$ for every $t \geq 0$ and $j,k \in \{ 1,\ldots,d \}$ via the so called stochastic sewing lemma and deriving \eqref{young1} by means of approximation arguments and stability properties of Young's integrals.	
Combining the arguments in Theorem \ref{Khoa} and Proposition \ref{cor 5.7}, we have the following representation of the Malliavin derivative of $X^x$
\begin{cor}[Representation of the Malliavin derivative of $X^x$] Suppose Assumption \ref{asumb2}. For any $\beta' \in (0,\beta)$, the Malliavin derivative of the solution $X_t^x$ to SDE \eqref{sde} is given by:
	\begin{align}\label{young2}
		(D_s^iX_t^x)^{j}= \delta_{i,j} + \sum_{k=1}^{d}\int_{s}^{t}(D_s^iX_u^x)^{j}\mathrm{d}\boldsymbol{V}_s^{k,j}(b,X), \quad \forall i,j\in \{1,\cdots,d\}.
	\end{align}
	Moreover, the map $s\mapsto D_s X^x(\cdot)$ is a.s. $\frac{1+\beta'}{2}\text{--H\"{o}lder}$ continuous for every $\beta' \in (0,\beta)$. 
\end{cor}
Although the validity of the process $\boldsymbol{V}_t^{k,j}(b,X)= \mathcal{A}_t^X[\partial_k b^j]$ in equation \eqref{young2} is given in \cite{Khoa}, one still need the following stability result to establish \eqref{young2}
\begin{lemm}
	Let $(b_n)_{n\geq 1}$ be a sequence of compactly supported smooth functions  approximating $b$ and let $(X^{n,x})_{n\geq 1}$ be a sequence of corresponding solution to the SDE \eqref{sde}. Then we have the following stability result: For all $p\geq 2$
	\begin{align}\label{5.13}
		\lim_{n\rightarrow \infty}\sup_{x\in \mathbb{R}^d}\mathbb{E}\Big[ \sup_{0\leq s \leq t \leq T}|D_sX_t^{n,x} - D_sX_t^{x} |^p   \Big] = 0.
	\end{align}
\end{lemm}
\begin{proof}
	To establish \eqref{5.13}, it is enough to prove that
	\begin{align}
		\lim_{n\rightarrow \infty}\sup_{x\in \mathbb{R}^d}\mathbb{E}\Big[ \sup_{0\leq s \leq t \leq T}|D_s\tilde{X}_t^{n,y} - D_s\tilde{X}_t^{y} |^p   \Big] = 0.
	\end{align}
	where $\tilde{X}$ satisfies \eqref{auxi}, while $\tilde{X}^n$ solves an SDE similar to the equation\eqref{auxi} with coefficients $\tilde{b}^n$ and $\tilde{\sigma}^n$ in lieu of $\tilde{b}$ and $\tilde{\sigma}$ respectively,  such that $\lim_{n\rightarrow
		\infty} (\tilde{b}^n,\tilde{\sigma}^n)= (\tilde{b},\tilde{\sigma})$. Therefore, both $D_s\tilde{X}_t^{n,y}$ and $D_s\tilde{X}_t^{y}$ satisfy a linear SDE with bounded coefficients. By applying classical result in the theory of SDE, one can obtained the sought result. This ends the proof.
\end{proof}
We have the subsequent results as a combination of Theorem \ref{theoGFP} and Theorem \ref{diffY}.

\begin{thm}\label{ClassYM}
Suppose Assumption \ref{asumb2}, (AY) and (AY1) hold. Let $x\in \mathbb{R}^d$ and $p > 1$. Then the map $x\mapsto (Y^x,Z^x)$ solution to \eqref{bsde3} is differentiable in the norm topology and the derivative process $(\nabla_x Y^x,\nabla_x Z^x)$ solves the BSDE
\begin{align}\label{5.15}
	\nabla Y^x_t =& \nabla_x \phi(X_T^x)\nabla X_T^x + \int_{t}^{T} \langle \nabla g(s,\boldsymbol{X}^x_s),\nabla \boldsymbol{X}_s^x \rangle \mathrm{d}s - \int_{t}^{T} \nabla_x Z_s^x \mathrm{d}B_s,
\end{align}
where $\nabla_x X^x$ satisfies equation \eqref{young1}. 
\end{thm}

\begin{thm}\label{Main1}
Suppose Assumptions \ref{asumb2}, (AY) and (AY1) are in force. Then the solution process $(Y^x,Z^x)$ belongs to $\mathbb{L}_{1,2} \times (\mathbb{L}_{1,2})^d$ and a version of $(D_uY_t^x,D_uZ_t^x)_{u,t\in [0,T]}$ is the unique solution to the BSDE \eqref{5.7}
with $D_uX^x$ now satisfies equation \eqref{young2}.

Moreover, $\{ D_tY_t^x: 0\leq t \leq T \}$ is continuous a version of $\{ Z_t^x: 0\leq t \leq T\}$ and 
\begin{align}
	D_u Y_t^x (\nabla_x X_u^x) &= \nabla_x Y_t^x, \label{eq 6.15} \\
	Z_t^x (\nabla_x X_t^x) &= \nabla_x Y_t^x, \label{eq 6.16}\\
	D_u Z_t^x (\nabla_x X_t^x) &= \nabla_x Z_t^x \label{eq 6.17}
\end{align}
\end{thm}
\begin{proof}

The first part follows from Theorem \ref{Main}

For the second part, it is enough to establish \eqref{eq 6.15} under the additional assumption on $b$. The relation \eqref{eq 6.16} follows from \eqref{eq 6.15} for $u=t$. We also have
\begin{align*}
	D_u Y_t^x\nabla_x X_u^x &= \nabla_x \phi(X_T^x)D_uX_T^x \nabla_x X_u^x  -\int_t^T D_uZ_s^x \nabla_xX_u^x \mathrm{d}B_s\\
	&+\int_t^T \langle \nabla g(s,\boldsymbol{X}_s^x), D_u\boldsymbol{X}_s^x \nabla_xX_u^x \rangle \mathrm{d}s
\end{align*}
with terminal value and generator satisfying conditions (C1),(C2) and (C3). Then, from the unique solvability of BSDE \eqref{bsde1} we obtain \eqref{eq 6.15} and \eqref{eq 6.17}. 
Moreover, $\{ D_tY_t^x: 0\leq t \leq T \}$ is a version of $\{ Z_t^x: 0\leq t \leq T\}$. We end the proof by observing that the process $ DY$ admits a continuous version, since it is represented in terms of continuous processes $t\rightarrow \nabla Y_t$ and $t\rightarrow(\nabla_x X_t)^{-1}$ (see Theorem \ref{ClassYM} and Theorem \ref{Khoa}). Thus, the existence of a continuous version to $t \mapsto Z_t$ follows.
\end{proof}

\section{Path regularity and explicit convergence rate}\label{NA}
In this section, we study the path regularity and the rate of convergence of a numerical scheme to the FBSDE \eqref{sde}-\eqref{bsde3}.
\subsubsection{Path regularity and Zhang's approximation}
In the sequel, $\Delta_N$ stands for the collection of all partitions of the interval $[0,T]$ by finite families of real numbers. Particular partitions will be denoted by $\delta_N= \{ t_i: 0= t_0 <\cdots < t_N=T \}$ with $N\in \mathbb{N}$. We define the mesh size of partitions as $|\delta_N| =\max_{0\leq i\leq N}|t_{i+1}-t_i|$
\begin{thm}[Path regularity] \label{thmpath1}Suppose Assumption \ref{asumb2} and (AY) be in force. Then for all $p \geq 2$, there exists a positive $C_p$ such that for any partition $\delta_N$ of $[0,T]$ with $(N+1) \in \mathbb{N}$ points and mesh size $|\delta_N|$, such that
\begin{align*}
	\sum_{i=0}^{N-1} \mathbb{E}\Big[ \Big( \int_{t_i}^{t_{i+1}} |Z_t-Z_{t_i}|^2\mathrm{d}t \Big)^{p/2} \Big] \leq C(p)|\delta_N|^{p/2}.
\end{align*}
\end{thm}
\begin{proof}\leavevmode
 We first assume that the functions $g$ and $\phi$ are continuously differentiable in $x,y$ and $z$.
For $t\in [t_i, t_{i+1}]$, using the representation \eqref{eq 6.16}, we have
\begin{align*}
	Z_t -Z_{t_i} &= \nabla Y_t(\nabla_xX_t)^{-1} - \nabla Y_{t_i}(\nabla_xX_{t_i})^{-1}\\
	&= \Big( \nabla Y_t -\nabla Y_{t_i} \Big)(\nabla_xX_{t_i})^{-1}  + \nabla Y_{t}\Big( (\nabla_xX_t)^{-1} -(\nabla_xX_{t_i})^{-1}  \Big)\\
	&= I_1 + I_2.
\end{align*}
Thus from H\"{o}lder's inequality, we have
\begin{align*}
	\mathbb{E}\Big[ \Big( \int_{t_i}^{t_{i+1}} |Z_t- Z_{t_i}|^2 \mathrm{d}t \Big)^{p/2} \Big] &\leq |\delta_N|^{p/2-1} \int_{t_i}^{t_{i+1}} \mathbb{E}\Big[|Z_t- Z_{t_i}|^{p}\Big] \mathrm{d}t \\
	&\leq |\delta_N|^{p/2-1} \int_{t_i}^{t_{i+1}} \mathbb{E}\left[ |I_1|^{p} + |I_2|^{p}  \right] \mathrm{d}t.
\end{align*} 

Let us first focus on the term  $I_2$. We first assume that $b$ is a bounded and compactly supported smooth function.  Indeed, if $b$ is not differentiable, then, by denseness of the set of compactly supported and differentiable functions in the set of bounded functions, there exists a sequence $(b_n)_{n\ge 1}$ of compactly supported and smooth functions converging to $b$  a.e. on $[0, T]\times \mathbb{R}^d$ and the desired result is obtained from the Vitali's convergence theorem. 
Then, the map $t\mapsto (\nabla_x X_t)^{-1}$ satisfies the following linear equation :
\begin{align*}
	\begin{cases}
		\displaystyle \frac{\mathrm{d} (\nabla_x X_t)^{-1}}{\mathrm{d}t} &= -(\nabla_x X_t)^{-1} b'(t,X_t),\\
		(\nabla_x X_0)^{-1} &= I_{d}.
	\end{cases}
\end{align*}
Iterating the above equation gives for all $0 \leq s \leq t$  
\[  |(\nabla_x X_t)^{-1} - (\nabla_x X_s)^{-1}| = \Big|\sum_{k=1}^{\infty} \int_{s<s_1<\ldots<s_n<t}  b'(s_1,X_{s_1}^x):\cdots: b'(s_k,X_{s_k}^x)\mathrm{d}s_1\cdots\mathrm{d}s_k \Big|,\]
where the symbol $":"$ stands for the matrix multiplication.

For any $p\geq 2$, the Girsanov's theorem and H\"{o}lder's inequality yield 
\begin{align*}
	&\mathbb{E}|(\nabla_x X_t)^{-1} - (\nabla_x X_s)^{-1}|^p\\
	&= \mathbb{E}\Big[ \Big| \sum_{k=1}^{\infty} \int_{s<s_1<\ldots<s_n<t}  b'(s_1,B_{s_1}^x):\cdots: b'(s_k,B_{s_k})\mathrm{d}s_1\cdots\mathrm{d}s_k  \Big|^p \times \mathcal{E}\Big( \int_{0}^{t} b(v,B_v)\mathrm{d}v  \Big)\Big]\\
	&\leq C(\|b\|_{\infty}) \mathbb{E}\Big[ \Big| \sum_{k=1}^{\infty} \int_{s<s_1<\ldots<s_n<t}  b'(s_1,B_{s_1}^x):\cdots: b'(s_k,B_{s_k})\mathrm{d}s_1\cdots\mathrm{d}s_k  \Big|^{2p} \Big]^{1/2}.
\end{align*}
Using \cite[Proposition 3.7]{ MMNPZ13},we have 
\[ \mathbb{E}|(\nabla_x X_t)^{-1} - (\nabla_x X_s)^{-1}|^p \leq C(\|b\|_{\infty}) |t-s|^{p/2}  \]
for all $p \geq 2$, where $C:[0,\infty)\rightarrow [0,\infty)$ is an increasing,continuous function, $\|\cdot\|$ is a matrix-norm on $\mathbb{R}^{d\times d}$ and $\|\cdot\|_\infty$ the supremum norm. Thus, the bound remains valid for only bounded and measurable drift $b.$
Therefore, 
\begin{align*}
	\mathbb{E}\left[ |I_2|^{p}  \right] &\leq C \Big( \mathbb{E}\Big[\sup_{0\leq t\leq T} |\nabla_x Y_t|^{2p}\Big]  \Big)^{\frac{1}{2}} \times \Big( \mathbb{E}\Big[| (\nabla_x X_t)^{-1} -(\nabla_x X_{t_i})^{-1}|^{2p}\Big]  \Big)^{\frac{1}{2}}\\
	&\leq C |\delta_N|^{p/2}.
\end{align*}
Let us turn now on $I_1$. We also claim:
\begin{align*}
	\delta_N^{p/2-1}\sum_{i=0}^{N-1}\int_{t_i}^{t_{i+1}} \mathbb{E}\left[ |I_1|^{p}\mathrm{d}t\right]\leq C |\delta_N|^{p/2} .
\end{align*}
Indeed, noticing that $(\nabla_x X_{t_i})^{-1}$ is $\mathfrak{F}_{t_i}\text{-adapted}$ and using the tower property
\begin{align*}
	\mathbb{E}[|(\nabla_x Y_t - \nabla_xY_{t_i}) (\nabla_x X_{t_i}^x)^{-1}|^p] = \mathbb{E}\left[ \mathbb{E}[|\nabla_x Y_t -\nabla_x Y_{t_i}|^{p}/\mathfrak{F}_{t_i}]|(\nabla_x X_{t_i})^{-1}|^{p}\right]
\end{align*}
By writing the equation satisfied by the difference $\nabla Y_t - \nabla Y_{t_i}$ for all $t_{i} \leq t \leq t_{i+1}$ and using the conditional BDG's inequality, we obtain
\begin{align*}
	&\mathbb{E}\Big[|\nabla_x Y_t -\nabla_x Y_{t_i}|^{2p}/\mathfrak{F}_{t_i} \Big]\\
	\leq& C \mathbb{E}\Big[ \Big| \int_{t_i}^{t} \langle \nabla g(s,\boldsymbol{X}_s), \nabla \boldsymbol{X}_s  \rangle \mathrm{d}s \Big|^{2p} +  \Big|\int_{t_i}^{t} \nabla Z_s \mathrm{d}B_s \Big|^{2p} \Big/\mathfrak{F}_{t_i} \Big] \\
	\leq	& C \mathbb{E}\Big[ \Big( \int_{t_i}^{t_{i+1}} |\langle \nabla g(s,\boldsymbol{X}_s), \nabla \boldsymbol{X}_s\rangle| \mathrm{d}s \Big)^{2p} +  \Big(\int_{t_i}^{t_{i+1}} |\nabla Z_s|^2 \mathrm{d}s \Big)^{p} \Big/\mathfrak{F}_{t_i}  \Big] :=C\X_{[t_i, t_{i+1}]}.
\end{align*}
Thus,
\begin{align*}
	&\delta_N^{p/2-1}\sum_{i=0}^{N-1}\int_{t_i}^{t_{i+1}} \mathbb{E}\left[ |I_1|^{p}\right]\mathrm{d}t\\
	&\leq C|\delta_N|^{p/2}\sum_{i=0}^{N-1} \mathbb{E}\left[ \X_{[t_i, t_{i+1}]} |(\nabla_x X_{t_i})^{-1}|^{p}\right]\\
	&\leq C |\delta_N|^{p/2} \mathbb{E}\Big[ \sup_{0\leq t \leq T}|(\nabla_x X_{t})^{-1}|^{p} \sum_{i=0}^{N-1}\X_{[t_i, t_{i+1}]}\Big]\\
	&\leq C |\delta_N|^{p/2}\mathbb{E}\Big[ \sup_{0\leq t \leq T}|(\nabla_x X_{t})^{-1}|^{p} \X_{[0,T]}\Big]
\end{align*}
The proof is complete if we prove that $\X_{[0,T]}$ is finite almost surely, since $\mathbb{E} \sup_{0\leq t \leq T}|(\nabla_x X_{t})^{-1}|^{p} < \infty$, thanks to Remark \ref{kun}. We recall that 
\begin{align*}
\X_{[0,T]} = \mathbb{E}\Big[ \Big( \int_{0}^{T} |\langle \nabla g(t,\boldsymbol{X}_t), \nabla \boldsymbol{X}_t \rangle| \mathrm{d}t \Big)^{2p} +  \Big(\int_{0}^{T} |\nabla Z_t|^2 \mathrm{d}t \Big)^{p}  \Big]:= II_1 + II_2, 	
\end{align*}
and $(\nabla Y, \nabla Z)$ is the unique solution to \eqref{5.15}. Then, from Lemma \ref{thap0}. for all $p \geq 2$ there exists $q \in (1,\infty)$ only depending on $\|Z*B\|_{BMO}$ such that 
\begin{align*}
	\mathbb{E} \left[ \sup_{0\leq t \leq T}  |\nabla Y_t|^{2p} + \Big( \int_0^T |\nabla Z_t|^2 \mathrm{d}t\Big)^p  \right] \leq \mathbb{E} \left[ |\nabla \phi(X_T) \nabla X_T
	|^{2pq} + \Big( \int_0^T |\nabla_x g(s,\boldsymbol{X}_t) \nabla_x X_t| \mathrm{d}t  \Big)^{2pq}   \right]^{\frac{1}{q}} .
\end{align*}
From the (AY) we deduce that, for any $\gamma \geq 2$
\begin{align*}
&\mathbb{E}\Big( \int_0^T |\nabla_x g(t,\boldsymbol{X}_t) \nabla_x X_t| \mathrm{d}t  \Big)^{\gamma} \\
&\leq \mathbb{E}\Big( \int_0^T \Lambda_x(1 +|Y_t| + f(|Y_t|)|Z_t|^{\alpha}) |\nabla_x X_t| \mathrm{d}t  \Big)^{\gamma}\\
&\leq C  \mathbb{E}\Big( \sup_{0\leq t\leq T}|\nabla_x X_t| \int_0^T (1 +|Y_t| + |Z_t|^2) \mathrm{d}t  \Big)^{\gamma}\\
&\leq C  \mathbb{E}( \sup_{0\leq t\leq T}|\nabla_x X_t|^{2\gamma} ) + C  \mathbb{E}\Big( \int_0^T (1 +|Y_t| + |Z_t|^2) \mathrm{d}t  \Big)^{2\gamma}
\end{align*}
where we used the local boundedness of $f$, the bound of the process $Y$ and the inequality $|ab| \leq a^2 + b^2$ to obtain the last inequality. From Theorem \ref{theoGFP}  and Lemma \ref{thap0}, we deduce the existence of a constant $C$ only depending on the coefficients appearing in assumption (AY), such that for any $\gamma \geq 2$
\[  \mathbb{E}\Big( \int_0^T |\nabla_x g(t,\boldsymbol{X}_t) \nabla_x X_t| \mathrm{d}t  \Big)^{\gamma} 	\leq C \]
and this implies $II_2 \leq C(1 + \Lambda_{\phi}^{2p})$. Let us now focus on the term $II_1.$ From the assumptions of the theorem again, we deduce
\begin{align*}
	&II_1\\
	&\leq \mathbb{E}\left( \int_{0}^{T} \big(|\nabla_x g(t,\boldsymbol{X}_t)| |\nabla_x{X}_t| + |\nabla_y g(t,\boldsymbol{X}_t)| |\nabla_x{Y}_t| + |\nabla_z g(t,\boldsymbol{X}_t)| |\nabla_x{Z}_t| \big) \mathrm{d}t \right)^{2p} \\
	&\leq C + \mathbb{E}\left( \int_{0}^{T} \big(\Lambda_y(1+ |Z_t|^{\alpha}) |\nabla_x{Y}_t| + \Lambda_z(1+ f(|Y_t|)|Z_t|) |\nabla_x{Z}_t| \big) \mathrm{d}t \right)^{2p}\\
	&\leq C + \mathbb{E}\left( \sup_{0\leq t\leq T}|\nabla_x{Y}_t|^2 + \Big(\int_{0}^{T}|Z_t|^2\mathrm{d}t\Big)^2 + \int_{0}^{T}|\nabla_x Z_t|^2\mathrm{d}t   \right)^{2p}\\
	&\leq C (1 + \Lambda_{\phi}^{2p})
\end{align*}
\end{proof}
Below, under much weaker assumptions ( the drift $b$ is only H\"older continuous) we provide the celebrated Zhang's approximation theorem
\begin{cor}[Zhang's path regularity theorem]
Under the assumptions of Theorem \ref{thmpath1} , we deduce the following:
\begin{align*}
	\sum_{i=0}^{N-1} \mathbb{E}\Big[ \int_{t_i}^{t_{i+1}} |Z_t-\tilde{Z}_{t_i}^{\delta_N}|^2\mathrm{d}t \Big] \leq C|\delta_N|,
\end{align*}
where
\begin{equation}\label{1.4}
	\tilde{Z}_{t_i}^{\delta_N} = \frac{1}{t_{i+1} - t_i}\mathbb{E}\Big[ \int_{t_{i}}^{t_{i+1}} Z_s \mathrm{d}s \Big/ \mathfrak{F}_{t_i} \Big],
\end{equation}
is a family of random variables defined for all partition points $t_i$ of $\delta_N$ and $Z$ is the control process in the solution of FBSDE \eqref{sde}-\eqref{bsde3}.
\end{cor}
\begin{proof}
It is well known that the random variable $	\tilde{Z}_{t_i}^{\delta_N}$ is the best $\mathfrak{F}_{t_i}\text{-adapted } \mathcal{H}^2([t_i,t_{i+1}]) $ approximation of $Z$, i.e., 
\begin{align*}
	\mathbb{E}\Big[ \int_{t_{i}}^{t_{i+1}} |Z_s-\tilde{Z}_{t_i}^{\delta_N}|^2 \mathrm{d}s  \Big] = \inf_{\tiny Z_i \in L^2(\Omega,\mathfrak{F}_{t_i})} \mathbb{E}\Big[ \int_{t_{i}}^{t_{i+1}} |Z_s-Z_{i}|^2 \mathrm{d}s \Big]. 
\end{align*}
In particular,
\begin{align*}
	\mathbb{E}\Big[ \int_{t_{i}}^{t_{i+1}} |Z_s-\tilde{Z}_{t_i}^{\delta_N}|^2 \mathrm{d}s  \Big] \leq \mathbb{E}\Big[ \int_{t_{i}}^{t_{i+1}} |Z_s-Z_{t_i}|^2 \mathrm{d}s  \Big].
\end{align*}
The result then follows from the Theorem \ref{thmpath1} by taking $p=2$.
\end{proof}
\begin{lemm}\label{lempath1}
	Under Assumptions \ref{asumb2} and (AY), we obtain for all $p> 1$
	\begin{align}
		\mathbb{E}	\Big[ \sup_{0\leq t\leq T } |Z_t|^{2p}  \Big] &< \infty. \label{1.2}
	\end{align}
	In addition for all $p \geq 2$, there exists a positive constant $C_p > 0$, such that for $0\leq s\leq t \leq T$
	\begin{align*}
		\mathbb{E}\Big[ \sup_{s\leq r\leq t }|Y_r-Y_s|^{p}  \Big]\leq C_p |t-s|^{p/2}.
	\end{align*}
\end{lemm}
\begin{proof} Under the assumptions of the Lemma, it follows from Lemma \ref{Main1} that the process $\{\nabla_x Y_t (\nabla_x X_s)^{-1}:0\leq s\leq t\leq T \}$ is a version of $\{ D_sY_t:0\leq s\leq t\leq T\}$. Using H\"{o}lder inequality and Remark \ref{kun}, for any $p > 1$ we deduce that: 
	\begin{align*}
		\mathbb{E}[\sup_{0\leq s\leq t\leq T}|D_sY_t|^{2p}]&\leq  \mathbb{E}\Big[ \sup_{0\leq t\leq T} |\nabla_x Y_t|^{4p} \Big]^{\frac{1}{2}}	\mathbb{E}\Big[ \sup_{0\leq s\leq T} |(\nabla_x X_s)^{-1}|^{4p} \Big]^{\frac{1}{2}}\\
		&\leq \|\nabla_x Y\|_{\mathcal{S}^{2p}} \mathbb{E}\Big[ \sup_{0\leq s\leq T}|(\nabla_x X_s)^{-1}|^{4p} \Big]^{\frac{1}{2}}<\infty,
	\end{align*}
	In particular, for $s=t$ we obtain the bound \eqref{1.2}.
	On the other hand, we recall that for all $ s\leq v \leq t $,
	\[ Y_s = \phi(X_t) + \int_s^t g(v,X_v,Y_v,Z_v)\mathrm{d}v -\int_s^t Z_v \mathrm{d}B_v. \]
	Then from BDG inequality we deduce that 
	\begin{align*}
		\mathbb{E}\Big[ \sup_{s\leq r\leq t }|Y_r-Y_s|^{p}  \Big]\leq C(p) \Big\{ \mathbb{E}\Big( \int_s^t |g(v,X_v,Y_v,Z_v)|\mathrm{d}v  \Big)^{p}  + \mathbb{E}\Big( \int_{s}^{t} |Z_v|^2 \mathrm{d}v\Big)^{p/2}   \Big\}.
	\end{align*}
	By using the bound: $|g(v,X_v.Y_v.Z_v)|\leq K(1+ |Y_v| + (1+f(|Y_v|))|Z_v|^2)$ and the fact that $Y_v$ is bounded we obtain that:
	\begin{align*}
		\mathbb{E}\Big[ \sup_{s\leq r\leq t }|Y_r-Y_s|^{p}  \Big] \leq C(p)\Big\{ |t-s|^{p} + \mathbb{E}\Big[ \Big( \int_{s}^{t} (1+ f(|Y_v|))|Z_v|^2\mathrm{d}v \Big)^{p} + \Big( \int_{s}^{t} |Z_v|^2 \mathrm{d}v\Big)^{p/2}  \Big]  \Big\}.
	\end{align*} 
	From to the local boundedness of $f$ and the bound \eqref{1.2}, we deduce the following
	\begin{align*}
		&\mathbb{E}\Big[ \sup_{s\leq r\leq t }|Y_r-Y_s|^{p}  \Big]\notag\\
		\leq& C(p)\Big\{ |t-s|^{p} + |t-s|^{p}\mathbb{E} [\sup_{s\leq v\leq t}|Z_v|^{2p}] +|t-s|^{p/2}\mathbb{E} [\sup_{s\leq v\leq t}|Z_v|^{2p}]  \Big\}\\
		\leq	& C(p)\Big\{ |t-s|^{p} + |t-s|^{p/2} \Big\}.
	\end{align*}
	The proof is completed.
\end{proof}
\subsubsection{Numerical approximation}
Let us introduce the following family of truncated FBSDE
\begin{align}\label{TruncY}
Y_t^n = \phi(X_T) + \int_{t}^{T}g_n(s,X_s,Y_s^n,Z_s^n)\mathrm{d}s -\int_{t}^{T} Z_s^n\mathrm{d}B_s,
\end{align}
where \[ g_n(t,x,y,z) = g(t,x,\tilde{\rho}_n(y),{\rho}_n(z)) \text{ for } (t,x,y,z) \in  [0,T]\times \mathbb{R}^d\times \mathbb{R} \times \mathbb{R}^d,\,\, n\in \mathbb{N}, \]
with $\tilde{\rho}_n$ given by \eqref{trunc}, $\rho_n: \mathbb{R}^d \rightarrow \mathbb{R}^d, z\mapsto \rho_n(z)= (\tilde{\rho}_n(z_1),\cdots,\tilde{\rho}_n(z_d))$ and $X$ stands for the solution to the SDE \eqref{sde}. It is clear that \eqref{TruncY} satisfies assumptions of Theorem \ref{Main} provided \eqref{sde} and \eqref{bsde3} do satisfy them. Using similar arguments as in the proof of Theorem \ref{th1} and Lemma \ref{lemma 5.1}, we have 
\begin{align}
\max \Big\{ \sup_{n\in \mathbb{N}} \|Z^n*B\|_{BMO}, \|Z*B\|_{BMO}  \Big\} \leq \Upsilon^{(2)}.
\end{align}
Thus for all $n\in \mathbb{N}$ the sequence $(Z^n*B)_{n\in \mathbb{N}}$ satisfies the properties (P2) of Lemma \ref{lem 2.1} i.e. there is a universal constant $r> 1$ such that $\mathcal{E}(Z^n*B) \in L^r$. Then, we deduce that $(Y^n,Z^n)$ is differentiable in the sense of Theorem \ref{ClassYM} and the following uniform bounds hold: 
$ \sup_{n \in \mathbb{N}} \Big(\|\nabla Y^n\|_{\mathcal{S}^{2p}} + \|\nabla Z^n\|_{\mathcal{H}^{2p}}\Big) < \infty $. Moreover, we can show as in \eqref{1.2} that
\begin{align}\label{ineq 6.7}
\sup_{n \in \mathbb{N}}  \mathbb{E} \big[ \sup_{t \in [0,T]} |Z_t^n|^{2p} \big] < \infty.
\end{align}
Furthermore, the properties of $\tilde{\rho}$ yield
\begin{equation}\label{ineq 6.8}
\Big(|\tilde{\rho}(Y_s^n)-Y_s^n| + |{\rho}(Z_s^n)-Z_s^n|\Big)^2\leq 8 ((\Upsilon^{(1)})^2 \mathbb{I}_{\{ |Y^n_s| > n\}} +|Z_s^n|^2\mathbb{I}_{\{ |Z^n_s| > n\}}) ,
\end{equation}
where $\sup_{n \in \mathbb{N}} |Y_s^n| \leq \Upsilon^{(1)}$( see Theorem \ref{th1}). 
Indeed, from definition of the function $\tilde{\rho}$ we deduce the following:
\begin{align*}
|\tilde{\rho}(Y_s^n)-Y_s^n| = \mathbb{I}_{\{|Y_s^n|>n\}}|\tilde{\rho}(Y_s^n)-Y_s^n| + \mathbb{I}_{\{|Y_s^n|\leq n\}}|\tilde{\rho}(Y_s^n)-Y_s^n| \leq 2|Y_s^n| \mathbb{I}_{\{|Y_s^n|>n\}} 
\end{align*}
Using the same reasoning as above we deduce that $|{\rho}(Z_s^n)-Z_s^n|\leq 2|Z_s^n| \mathbb{I}_{\{|Z_s^n|>n\}}$. 

Below we provide the convergence error of the truncation 
\begin{thm}\label{thmconve1}
Let assumptions of Theorem \ref{Main} be in force. Let $(X,Y,Z)$ be the solution to equation \eqref{sde}-\eqref{bsde3} and $(X,Y^n,Z^n)$ be the solution of \eqref{sde}-\eqref{TruncY}, $n\ge 1$. Then for any $p > 1$ and $\kappa \geq 1$ there exist a positive finite constant $C(p,k)$ depending on $p,\kappa,$, \eqref{ineq 6.7}, the bound in Lemma \ref{lemma 5.1} and $\Upsilon^{(1)}$ such that for $n\in \mathbb{N}$
\begin{align*}
	\mathbb{E}\Big[ \sup_{t\in [0,T]} |Y_t^n -Y_t|^{2p} + \Big( \int_{0}^{T} |Z^n_s -Z_s|^2 \mathrm{d}s\Big)^p  \Big] \leq C(p,\kappa) (n)^{\frac{-\kappa}{4q}}
\end{align*}
\end{thm}
\begin{proof} Let $\zeta^n$ be given by		
\[ \zeta_t^n:= \frac{g(t,X_t,Y_t^n,Z_t^n)-g(t,X_t,Y_t^n,Z_t)}{|Z_t^n- Z_t|^2} (Z_t^n- Z_t)1_{\{ Z_t^n- Z_t \neq 0\} }.\]
Then thanks to the assumptions of the theorem, $\zeta_t^n*B \in BMO(\mathbb{P})$.  In addition for some constant $r > 1$, independent of $n$ we have $\mathcal{E}(\zeta^n*B)^{-1} \in L^r(\mathbb{Q}^n)$.
Let $\varPi $ be define by $\varPi := \max\{ \|\mathcal{E}(\zeta^n*B) \|_{L^{r}(\mathbb{P})}, \|\mathcal{E}(\zeta^n*B)^{-1} \|_{L^{r}(\mathbb{Q}^n)} \}$. 
Then, from the Bayes' rule and  H\"{o}lder's inequality, we obtain for all $p> 1$ that 
\begin{align}
&\mathbb{E}\Big[ \sup_{t\in [0,T]}|Y_t^n -Y_t|^{2p} + \Big( \int_0^T  |Z_t^n -Z_t|^2 \mathrm{d}t\Big)^p  \Big]\notag\\
\leq &\varPi \Big[\mathbb{E}^{\mathbb{Q}^n}\Big( \sup_{t\in [0,T]}|Y_t^n -Y_t|^{2pq} + \Big( \int_0^T  |Z_t^n -Z_t|^2 \mathrm{d}t\Big)^{pq} \Big)\Big]^{\frac{1}{q}},
\end{align}
where $\mathrm{d}\mathbb{Q}^n := \mathcal{E}(\zeta^n*B)\mathrm{d}\mathbb{P}$.  Using Lemma \ref{thap}, we have
\begin{align*}
&\Big(\mathbb{E}^{\mathbb{Q}^n}\Big[ \sup_{t\in [0,T]}|Y_t^n -Y_t|^{2pq} + \Big( \int_0^T \! |Z_t^n -Z_t|^2 \mathrm{d}t\Big)^{pq} \Big]\Big)^{\frac{1}{q}} \\
\leq &C \Big(\mathbb{E}^{\mathbb{Q}^n}\Big[\int_0^T\!\! |g(s,X_s,Y_s^n,Z_s^n)-g(s,X_s,\tilde{\rho}(Y_s^n), {\rho}(Z_s^n))| \mathrm{d}s\Big]^{2pq} \Big)^{\frac{1}{q}}.\notag
\end{align*}
In addition for $K= \max(\Lambda_y,\Lambda_z)$
\begin{align*}
&|g(s,X_s,Y_s^n,Z_s^n)-g(s,X_s,\tilde{\rho}(Y_s^n), {\rho}(Z_s^n))| \\
\leq &K(1+2|Z_s^n|^{\alpha}+4\sup_{0\leq y\leq 2\Upsilon^{(1)}}\varphi(y)|Z_s^n|)\big(|\tilde{\rho}(Y_s^n)- Y_s^n| +|{\rho}(Z_s^n)-Z_s^n|\big).
\end{align*}

By using the above bound and applying  the H\"{o}lder's inequality, we deduce that
\begin{align*}
&\Big(\mathbb{E}^{\mathbb{Q}^n}\Big[\sup_{t\in [0,T]}|Y_t^n -Y_t|^{2pq} + \Big( \int_0^T  |Z_t^n -Z_t|^2 \mathrm{d}t\Big)^{pq} \Big]\Big)^{\frac{1}{q}} \notag\\	
&\leq C(q,p) \Big(\mathbb{E}^{\mathbb{Q}^n}\Big[\int_0^T \Big|K(1+2|Z_s^n|^{\alpha}+4\sup_{0\leq y\leq 2\Upsilon^{(1)}}\varphi(y)|Z_s^n|)\Big|^2\mathrm{d}s\Big]^{2pq}\Big)^{\frac{1}{2q}}\\
&\quad \times \Big(\mathbb{E}^{\mathbb{Q}^n}\Big[\int_0^T\Big(|\tilde{\rho}(Y_s^n)-Y_s^n| + |{\rho}(Z_s^n)-Z_s^n|\Big)^2 \mathrm{d}s\Big]^{2pq} \Big)^{\frac{1}{2q}}\notag
\end{align*}
Since $Z^n$ belongs  to $\mathcal{H}^{2p}$ and the function $\varphi$ is continuous, we deduce that the first term in the above inequality is uniformly bounded in $n$. It then remains to prove that the second term is also uniformly bounded. 
Since, $\mathbb{P}$ and $\mathbb{Q}^n$ are equivalent, using \eqref{ineq 6.7} and \eqref{ineq 6.8},  the H\"{o}lder and the Markov's inequalities give

\begin{align}
&\Big(\mathbb{E}^{\mathbb{Q}^n}\Big[\int_0^T\Big(|\tilde{\rho}(Y_s^n)-Y_s^n| + |{\rho}(Z_s^n)-Z_s^n|\Big)^2 \mathrm{d}s\Big]^{2pq} \Big)^{\frac{1}{2q}}\notag\\
\leq & C\Big(\mathbb{E}^{\mathbb{Q}^n}\Big[\int_0^T  ((\Upsilon^{(1)})^2 \mathbb{I}_{\{ |Y^n_s| > n\}} +|Z_s^n|^2\mathbb{I}_{\{ |Z^n_s| > n\}})  \mathrm{d}s\Big]^{2pq} \Big)^{\frac{1}{2q}}  \notag\\
\leq &C\Big( \int_0^T {\mathbb{Q}^n}{\{ |Y^n_s| > n\}}\mathrm{d}s \Big)^{\frac{1}{2q}} +  C\Big( \mathbb{E}^{\mathbb{Q}^n} \int_0^T |Z_s^n|^{8pq} \mathrm{d}s \Big)^{\frac{1}{4q}}\Big( \int_0^T {\mathbb{Q}^n} {\{ |Z^n_s| > n\}} \mathrm{d}s \Big)^{\frac{1}{4q}} \notag\\
\le& C(n)^{\frac{-\kappa}{2q}}\Big( \int_0^T \mathbb{E}^{\mathbb{Q}^n}\big[|Y_s^n|^{2\kappa}\big] \mathrm{d}s \Big)^{\frac{1}{2q}} + C (n)^{\frac{-\kappa}{4q}}\Big( \int_0^T \mathbb{E}^{\mathbb{Q}^n}|Z_s^n|^{8pq}\big)\mathrm{d}s \Big)^{\frac{1}{4q}}  \Big( \int_0^T \mathbb{E}^{\mathbb{Q}^n}\big[|Z_s^n|^{2\kappa}\big] \mathrm{d}s \Big)^{\frac{1}{4q}}\notag,
\end{align}
Using once more the uniform boundedness of $Y^n$, the Bayes' rule and the H\"{o}lder's inequality, we derive the existence of a constant $C$ that does not depend on $\kappa$ and $n$ such that:
\begin{align*}
&\Big(\mathbb{E}^{\mathbb{Q}^n}\Big[\int_0^T\Big(|\tilde{\rho}(Y_s^n)-Y_s| + |{\rho}(Z_s^n)-Z_s^n|\Big)^2 \mathrm{d}s\Big]^{2pq} \Big)^{\frac{1}{2q}}\notag\\	
&\leq C(n)^{\frac{-\kappa}{2q}} + C (n)^{\frac{-\kappa}{4q}} \Big(\int_0^T\mathbb{E}\Big[ \mathcal{E}(\Pi^n*B) |Z_s^n|^{2\kappa} \Big]\mathrm{d}s \Big)^{\frac{1}{4q}}\notag\\
&\leq  C(n)^{\frac{-\kappa}{2q}} +C \varPi (n)^{\frac{-\kappa}{4q}}\Big( \mathbb{E}\big[ \sup_{0\leq s \leq T}|Z_s^n|^{2\kappa q}\big] \Big)^{\frac{1}{4q^2}}.\notag
\end{align*}
This completes the proof.
\end{proof}

	\appendix
\section{Some Properties of BMO$\text{-Martingales}$}
The following result give some properties of $\text{BMO}$-martingales which will be  extensively used in this work. 
We refer the reader to \cite{Kazamaki} for more details on the subject.

\begin{lemm}\leavevmode
	\begin{itemize}\label{lem 2.1}
		\item[(P1)] Let $M$ be a BMO($\mathbb{P}$) martingale with quadratic variation $\langle M \rangle$. Let  $(\mathcal{E}(M)_t)_{0\leq t\leq T}$  be the process defined by
		$$\mathcal{E}(M)_t:= \exp \{ M_T -1/2 \langle M \rangle_t   \}.$$
		Then $\mathbb{E}[\mathcal{E}(M)_T]=1$ and the measure $\mathbb{Q}$ given by $\mathrm{d}\mathbb{Q} := \mathcal{E}(M)_T\mathrm{d}\mathbb{P}$ defines a probability measure.
		\item[(P2)] For a given  BMO($\mathbb{P}$) martingale $M,$ there exists $ r > 1$ such that $\mathcal{E}(M) \in L^{r}.$ Moreover, for any stopping time $\tau \in [0,T]$ it holds that
		\begin{align*}
			\mathbb{E}[\mathcal{E}(M)_T^{r} |\mathfrak{F}_{\tau}] \leq C(r) (\mathcal{E}(M)_{\tau})^{r},
		\end{align*} 
		\item[(P3)] If $\Vert M \Vert_{BMO} < 1,$ then for every stopping time $\tau \in [0,T]$
		\[ \mathbb{E}\left[ \exp\{ \langle M \rangle_T  - \langle M \rangle_{\tau} \} /\mathfrak{F}_{\tau} \right]  < \frac{1}{1- \Vert M \Vert_{BMO}^2} .\]
		\item[(P4)] 
		If $\int_0 Z_s\mathrm{d}B_s \in$  BMO($\mathbb{P}$) then for every $ p \geq 1$ it holds that $Z \in \mathcal{H}^{2p}(\mathbb{R}^d)$
		i.e	 \[ \mathbb{E}\Big[ \Big( \int_{0}^{T} |Z_s|^2 \mathrm{d}s    \Big)^p \Big] \leq p! \Big\Vert \int Z\mathrm{d}B \Big\Vert_{BMO}^{2p}. \]
		Moreover for any $p \geq 1$ and  $\epsilon \in (0,2)$
		\[ \mathbb{E}\Big[ \exp\Big( p\int_{0}^{T} |Z_s|^{\epsilon} \mathrm{d}s  \Big)\Big] \leq C,   \]
		where $C$ depends on $p,\epsilon$ and $\Vert \int Z\mathrm{d}B \Vert_{BMO}^{2}.$ 
	\end{itemize}
The following is taken from \cite{GeissYlinen}.
\begin{lemm}[Fefferman's inequality]\label{feffer}
	Let $(\mu_t)_{t\in [0,T]}$ and $(\nu_t)_{t\in [0,T]}$ be progressively measurable $\mathbb{R}\text{-valued}$ processes such that $\mathbb{E}\int_0^T|\nu_t|^2\mathrm{d}t < \infty$ and $p \in [1,\infty)$. Then, the following holds
	\begin{equation*}
		\left\|   \int_0^T |\mu_t \nu_t|\mathrm{d}t \right\|_p \leq c(p)\|\mu\|_{\mathcal{H}^p} \|\nu\|_{\mathcal{H}_{\text{BMO}(\mathbb{P})}}
	\end{equation*}
\end{lemm}
\end{lemm}
\section{Proofs of  Lemma \ref{thap} and Lemma \ref{thap2}}\label{proofs}
\begin{proof}[Proof of Lemma \ref{thap}] The proof follows as in \cite[Lemma A.1]{ImkRevRich}. For the sake of completeness, we briefly reproduce it here. Set $\delta Y = Y^1 - Y^2,$ $\delta Z = Z^1 - Z^2$ and $\delta \xi = \xi^1 -\xi^2, \delta g = g^1(\cdot,Y^2,Z^2) - g^2(\cdot,Y^2,Z^2).$ We also define :
	\begin{align*}
		\Gamma_t:&= \frac{g^1(t,Y_t^1,Z_t^1) - g^1(t,Y_t^2,Z_t^1)}{Y^1_t - Y^2_t}{1}_{\{Y^1_t - Y^2_t\neq 0\}}, \quad e_t:= \exp\Big( 2\int_{0}^{t} |\Gamma_s| \mathrm{d}s \Big),\\
		\Pi_t:&= \frac{g^1(t,Y_t^2,Z_t^1) - g^1(t,Y_t^2,Z_t^2)}{|Z_t^1 - Z_t^2|^2}(Z^1_t - Z^2_t){1}_{\{|Z^1_t - Z^2_t|\neq 0\}}.
	\end{align*}
	
	Using It\^{o}'s formula, and the Girsanov transform, the dynamics of the continuous semimartingale $(e_t \delta Y_t^2)_{t\geq 0}$  is given by 
	\begin{align*}
		\mathrm{d}[e_t(\delta Y_t)^2] 
		= (-2e_t\delta Y_t \delta g_t + e_t |\delta Z_t|^2)\mathrm{d}t + 2e_t\delta Y_t\delta Z_t \mathrm{d}B_t^\mathbb{Q},
	\end{align*}
	where, $B_{\cdot}^{\mathbb{Q}} = B_{\cdot} - \int_{0}^{\cdot}\Pi_s\mathrm{d}s$ is a {Brownian motion} under the probability measure $\mathbb{Q}$ with density $\mathrm{d}\mathbb{Q}/\mathrm{d}\mathbb{P} = \mathcal{E}(\int_{0}^{\cdot}\Pi_{\cdot}\mathrm{d}B_{\cdot})$. 
	Observe that $(e_t)_{t\in [0,T]}$ is strictly increasing and bounded from below by $1$, and for any $0\leq t \leq s\leq T$ 
	\[ e_s (e_t)^{-1} = \exp\Big\{ 2 \int_{t}^{s} |\Gamma_r|\mathrm{d}r \Big\} \leq A_T:= \exp\Big\{  2\int_{0}^{T} K(1+ |Z_s^1|^{\alpha}) \mathrm{d}s \Big\}.   \]
	Moreover since $Z^1\in \mathcal{H}_{\text{BMO}}$, $(e_t)_{t \in [0,T]} \in \mathcal{S}^p$ for all $p\geq 1$. Then, we deduce that
	\begin{align}\label{eqA6}
		e_t(\delta Y_t)^2 + \int_t^T e_s|\delta Z_s|^2 \mathrm{d}s 
		\leq e_T \delta \xi^2 
		+ 2\int_t^T e_s \delta Y_s \delta g_s \mathrm{d}s -  2 \int_t^T e_s\delta Y_s \delta Z_s \mathrm{d}B_s^\mathbb{Q} .
	\end{align}
	Let us first provide an estimate of the norm of $\delta Y_t$ under the measure $\mathbb{Q}$. Dividing both sides of \eqref{eqA6} by $e_t$, taking the conditional expectation on both sides and using the basic inequality $ab\leq \frac{1}{\epsilon}a^2+\epsilon b^2$ for $\epsilon>0$, we obtain
	\begin{align}
		\delta Y_t^2 \leq&  \mathbb{E}^{\mathbb{Q}} \Big[  \frac{e_T}{e_t} \delta \xi^2 +  2\int_0^T \frac{e_s}{e_t} \delta Y_s |\delta g_s| \mathrm{d}s   \Big|\mathfrak{F}_t \Big]\notag\\
		\leq	& \mathbb{E}^{\mathbb{Q}} \Big[ \frac{1}{\epsilon}\sup_{0\leq t \leq T}|\delta Y_t|^2 +  A_T^2 \X  \Big|\mathfrak{F}_t    \Big]  ,\label{ok}
	\end{align}
	where $\X$ is defined by  $\X:= |\delta\xi|^2 + \epsilon \Big(\int_{0}^{T}|\delta g_s|\mathrm{d}s\Big)^2$. Notice that $\mathbb{E}^{\mathbb{Q}} \Big[ \frac{1}{\epsilon}\sup_{0\leq t \leq T}|\delta Y_t|^2  \Big|\mathfrak{F}_t    \Big]$ (respectively $\mathbb{E}[A_T^2\X|\mathfrak{F}_t]$) is a right continuous martingale on $[0,T]$ with terminal random value given by $ \frac{1}{\epsilon}\sup_{0\leq t \leq T}|\delta Y_t|^2$ (respectively $A_T^2\X$). Then by the Doob's martingale inequality we have that for all $p > 1$
	\begin{align}
		\mathbb{E}^{\mathbb{Q}}\Big[\sup_{t\in [0,T]}\mathbb{E}^{\mathbb{Q}} \Big[ \frac{1}{\epsilon}\sup_{0\leq t \leq T}|\delta Y_t|^2  \Big|\mathfrak{F}_t    \Big] \Big]^{p}\leq& \Big( \frac{p}{p-1} \Big)^{p}\mathbb{E}^{\mathbb{Q}} \Big[ \frac{1}{\epsilon^{p}}\sup_{0\leq t \leq T}|\delta Y_t|^{2p}   \Big] ,\label{eqdoob1}\\
		\mathbb{E}^{\mathbb{Q}}\Big[\sup_{t\in [0,T]}\mathbb{E}^{\mathbb{Q}}[A_T^2\X|\mathfrak{F}_t] \Big]^{p}\leq& \Big( \frac{p}{p-1} \Big)^{p}\mathbb{E}^{\mathbb{Q}}[A_T^{2p}\X^p] . \label{eqdoob2}
	\end{align}
	Taking successively the absolute value, the $p\text{-th}$ power, the supremum and the expectation on both sides of \eqref{ok}  and using \eqref{eqdoob1} and \eqref{eqdoob2}, we obtain 
	\begin{align*}
		\mathbb{E}^{\mathbb{Q}}[\sup_{t\in [0,T]} |\delta Y_t|^{2p} ] \leq& \mathbb{E}^{\mathbb{Q}}\Big[\sup_{t\in [0,T]} \Big( \mathbb{E}^{\mathbb{Q}} \Big[ \frac{1}{\epsilon}\sup_{0\leq t \leq T}|\delta Y_t|^2 +  A_T^2 \X  \Big|\mathfrak{F}_t    \Big]\Big)^p  \Big]\\
		\leq &C(p) \Big(\mathbb{E}^{\mathbb{Q}} \Big[ \frac{1}{\epsilon^{p}}\sup_{0\leq t \leq T}|\delta Y_t|^{2p}   \Big]+\mathbb{E}^{\mathbb{Q}}[A_T^{2p}\X^p]\Big).
	\end{align*}
	Choosing $\epsilon$ such that $\frac{C(p)}{\epsilon^p}<1$ and using the H\"{o}lder's inequality for $\nu\geq 1$, we have
	\begin{align*}
		\mathbb{E}^{\mathbb{Q}}[\sup_{t\in [0,T]} |\delta Y_t|^{2p} ] \leq&C(p) \mathbb{E}^{\mathbb{Q}} \Big[ A_T^{2p} \X^p  \Big]\\
		\leq &C(p) \mathbb{E}^{\mathbb{Q}}[A_T^{\frac{2p\nu}{\nu-1}}]^{\frac{\nu-1}{\nu}}\mathbb{E}^{\mathbb{Q}}\Big[ |\delta\xi|^{2p\nu} +  \Big(\int_{0}^{T}|\delta g_s|\mathrm{d}s\Big)^{2p\nu}  \Big]^{\frac{1}{\nu}}.
	\end{align*}
	The first term on the right side of the above bound is finite (see Lemma \ref{lem 2.1}).
	
	Since $\Pi *B$ is a $BMO(\mathbb{P})$ martingale it follows that $-\Pi *B^{\mathbb{Q}}$ and $\Pi *B^{\mathbb{Q}}$ are  $BMO(\mathbb{Q})$ martingales. Then from Lemma \ref{lem 2.1} there exist $r,r_1> 1$ such that $\mathcal{E}(\Pi *B) \in L^r$ and $\mathcal{E}(-\Pi *B^{\mathbb{Q}}) \in L^{r_1}$. In addition, since $\mathcal{E}(\Pi *B)^{-1}= \mathcal{E}(-\Pi *B^{\mathbb{Q}})$, we have $\mathrm{d}\mathbb{P} = \mathcal{E}(-\Pi *B^{\mathbb{Q}})\mathrm{d}\mathbb{Q}$. Let $r'$ and $r_1'$ be the H\"{o}lder conjugates of $r$ and $r_1$, respectively. Then using Girsanov's theorem, we have 
	\begin{align*}
		\mathbb{E}[\sup_{t\in [0,T]} |\delta Y_t|^{2p} ] &= 	\mathbb{E}^{\mathbb{Q}}[ \mathcal{E}(-\Pi *B^{\mathbb{Q}})_T\sup_{t\in [0,T]} |\delta Y_t|^{2p} ] \\
		&\leq \mathbb{E}^{\mathbb{Q}}[ \mathcal{E}(-\Pi *B^{\mathbb{Q}})_T^{r_1}]^{\frac{1}{r_1}} \mathbb{E}^{\mathbb{Q}}[\sup_{t\in [0,T]} |\delta Y_t|^{2pr_1'} ]^{\frac{1}{r_1'}}\\
		&\leq C \mathbb{E}^{\mathbb{Q}}\Big[ |\delta\xi|^{2pr_1'\nu} + \Big(\int_{0}^{T}|\delta g_s|\mathrm{d}s\Big)^{2pr_1'\nu}  \Big]^{\frac{1}{r_1'\nu}}\\
		&\leq C \mathbb{E}[\mathcal{E}(\Pi *B)^r]^{\frac{1}{r}} \mathbb{E}\Big[ |\delta\xi|^{2pr_1'\nu r'} + \Big(\int_{0}^{T}|\delta g_s|\mathrm{d}s\Big)^{2pr_1'\nu r'}  \Big]^{\frac{1}{r_1'\nu r'}}
	\end{align*}
	We obtain the desired estimate for $\delta Y$ by taking $q= r_1'\nu r'$. Let us now provide the bound for $\delta Z$. From \eqref{eqA6} taking the $p\text{-th}$ power, applying the BDG inequality one can obtain:
	\begin{align*}
		\mathbb{E}^{\mathbb{Q}}\left( \int_t^T |\delta Z_s|^2 \mathrm{d}s\right)^p  \leq  \mathbb{E}^{\mathbb{Q}} |e_T\delta\xi^2|^p + 2\mathbb{E}^{\mathbb{Q}} \Bigg( \int_t^T  e_s \delta Y_s \delta g_s \mathrm{d}s \Bigg)^p + c(p)	\mathbb{E}^{\mathbb{Q}} \Bigg( \int_t^T  |e_s \delta Y_s \delta Z_s|^2 \mathrm{d}s \Bigg)^{\frac{p}{2}}
	\end{align*}
	By applying H\"older's inequality and using the fact that $1 \leq A_T \leq A_T^2$ we deduce that
	\begin{align*}
		\mathbb{E}^{\mathbb{Q}}\left( \int_t^T |\delta Z_s|^2 \mathrm{d}s\right)^p  \leq C \mathbb{E}^{\mathbb{Q}} \left[ A_T^{2p} \left( |\delta\xi|^{2p} + \epsilon \Big(\int_{0}^{T}|\delta g_s|\mathrm{d}s\Big)^{2p} \right) \right]	
	\end{align*}

	Similar techniques can be used to provide the bound under the measure $\mathbb{P}$. This conclude the proof.
\end{proof}
\begin{proof}[Proof of Lemma \ref{thap2}]
	We keep almost the same notations as in the above proof at the exception of the process $(e_t)_{t \in [0,T]}$ defined as follows: $e_t:= \exp\Big( \int_{0}^{t} |\Gamma_s| \mathrm{d}s \Big)$.

	Similar arguments as in the proof above gives:
	\begin{align*}
		|\delta Y_t| = \mathbb{E}^{\mathbb{Q}}\Big[  A_T \Xi_t |\mathfrak{F}_t \Big].
	\end{align*}
	where the process $\Xi$ is defined by  $\Xi_t:= |\delta\xi| + \int_{t}^{T}|\delta g_s|\mathrm{d}s$. Using $p_0 \in (1,p)$ we obtain from the reverse H\"{o}lder inequality that
	
	\begin{align*}
		|\delta Y_t| &\leq C(p_0')\mathbb{E}\Big[ (A_T \Xi_t)^{p_0} |\mathfrak{F}_t \Big]^{\frac{1}{p_0}}
	\end{align*}
	
	By Doob's maximal and the Cauchy-Schwartz inequalities
	\begin{align*}
		\mathbb{E}\left[ \sup_{s\in [t,T]}|\delta Y_s|^p \right] &\leq C(p_0')\Big(\frac{p}{p-p_0} \Big)^{\frac{p}{p_0}}\mathbb{E}\Big[ (A_T \Xi_t)^{p} \Big]\\
		&\leq C(p_0')\Big(\frac{p}{p-p_0} \Big)^{\frac{p}{p_0}}\mathbb{E}\Big[ (A_T)^{2p} \Big]^{\frac{1}{2}} \mathbb{E}\Big[ \Xi_t^{2p} \Big]^{\frac{1}{2}}.
	\end{align*}
	Thus, 
	\begin{align}
		\Big\| \sup_{s\in [t,T]}|\delta Y_s|  \Big\|_p \leq C(p,p_0',p_0) \| \Xi_t \|_{2p},
	\end{align}
from which we obtain the first inequality of the Lemma. Set
	$\Delta g_s= g^1(s,Y_s^1,Z_s^1) - g^2(s,Y_s^2,Z_s^2) $ and using the assumptions of the Lemma, we have 
	\begin{align*}
		|\Delta g_s| \leq |\delta g_s| + \Lambda_y \theta_s|\delta Y_s| + \Lambda_z \varphi_s |\delta Z_s|,
	\end{align*}
	where $\theta_s = (1 +2|Z_s^1|^{\alpha}) $ and $\varphi_s = (1 +2f(|Y_s^1|)(|Z_s^1|+ |Z_s^2|)) $ belong to $ \mathcal{H}_{\text{BMO}}.$ By applying the Fefferman's (see Lemma \ref{feffer}) inequality, the Young and H\"older inequalities, we deduce the following
	
	\begin{align*}
		\int_t^T |\delta Y_s \Delta g_s|\mathrm{d}s &\leq \int_t^T |\delta Y_s \delta g_s|\mathrm{d}s + \Lambda_y \int_t^T |\delta Y_s|^2 |\theta_s|\mathrm{d}s + \Lambda_z \int_t^T |\delta Y_s| |\varphi_s| |\delta Z_s|\mathrm{d}s\\
		&\leq \Lambda^2 \sup_{s\in [t,T]}|\delta Y_s|^2 + \frac{1}{2} \left( \int_t^T |\delta g_s|\mathrm{d}s\right)^2  + \Lambda_z \int_t^T |\delta Y_s| |\varphi_s| |\delta Z_s|\mathrm{d}s,
	\end{align*}
	where $\Lambda = \Big( \frac{1}{2} + \Lambda_y \|\theta_s\|_{\mathcal{H}_{\text{BMO}}} \Big)$.

Next, set $S_t(Z)^2:= \int_t^T |\delta Z_s|^2\mathrm{d}s$, and successively apply the It\^{o}'s formula, the Fefferman's inequality( Lemma \ref{feffer}) and BDG's inequality  to get
	
	\begin{align*}
		&	\| S_t(Z) \|_p \\
		&\leq \left\| \left(  |\delta \xi|^2 + 2\int_t^T |\delta Y_s||\Delta g_s|\mathrm{d}s + 2\left| \int_t^T \delta Y_s \delta Z_s\mathrm{d}B_s \right|    \right)^{\frac{1}{2}}  \right\|_p\\
		&\leq \left\| \left(  |\delta \xi|^2 + 2\Lambda^2 \sup_{s\in [t,T]}|\delta Y_s|^2 + \left( \int_t^T |\delta g_s|\mathrm{d}s\right)^2  + 2\Lambda_z \int_t^T |\delta Y_s| |\varphi_s| |\delta Z_s|\mathrm{d}s + 2\left| \int_t^T \delta Y_s \delta Z_s\mathrm{d}B_s \right|    \right)^{\frac{1}{2}}  \right\|_p\\
		&\leq \| \Xi\|_p + \sqrt{2}\Lambda \Big\| \sup_{s\in [t,T]}|\delta Y_s| \Big\|_p + \left( \sqrt{2\Lambda_z} \|\varphi_s\|_{\mathcal{H}_{\text{BMO}}}^{\frac{1}{2}} + \sqrt{2} c(p) \right) \left\| \Big(\int_t^T |\delta Y_s \delta Z_s|^2 \mathrm{d}s \Big)^{\frac{1}{2}} \right\|_{\frac{p}{2}}^{\frac{1}{2}}\\
		&\leq \| \Xi\|_p + (\sqrt{2}\Lambda + \kappa \sqrt{\frac{\varepsilon}{2}}) \Big\| \sup_{s\in [t,T]}|\delta Y_s| \Big\|_p + \frac{\kappa}{\sqrt{2\varepsilon}} \| S_t(Z) \|_p,
	\end{align*}
	for any $\varepsilon >0,$ and $\kappa = \left( \sqrt{2\Lambda_z} \|\varphi_s\|_{\mathcal{H}_{\text{BMO}}}^{\frac{1}{2}} + \sqrt{2} c(p) \right)$. Choosing for example $\varepsilon=\kappa ^2$ leads to the second bound	of the Lemma. The proof is completed.
\end{proof}

	\bibliographystyle{plain}
	\bibliography{Biblio, bibliography, Biblio1}

\begin{thebibliography}{10}

\bibitem{ArImDR}
S.~Ankirchner, P.~Imkeller, and G.~Dos Reis.
\newblock Classical and {V}ariational {D}ifferentiability of {BSDE}s with
  {Q}uadratic {G}rowth.
\newblock {\em Elect. J. Probab.}, 12:1418--1453, 2007.

\bibitem{Bahlali1}
K.~Bahlali.
\newblock Solving {U}nbounded {Q}uadratic {BSDE}s by a {D}omination method.
\newblock arXiv:1903.11325v1, 03 2019.

\bibitem{Ouknine}
K.~Bahlali, M.~Eddahbi, and Y.~Ouknine.
\newblock Quadratic {BSDE} with {$L^2$}-terminal data: {K}rylov's estimate and
  {I}t\^{o}-{K}rylov's formula and existence results.
\newblock {\em Annals of Probability}, 45:2377--2397, 2017.

\bibitem{BMBPD17}
D.~Banos, T.~Meyer-Brandis, F.~Proske, and S.~Duedahl.
\newblock Computing deltas without derivatives.
\newblock {\em Finance Stoch.}, 21(2):509--549, 2017.

\bibitem{BarrieuElkaroui}
P.~Barrieu and N.~El Karoui.
\newblock Monotone stability of quadratic semimartingales with applications to
  unbounded general quadratic {BSDE}s.
\newblock {\em Annals of Probability}, 14:1831--1863, 2013.

\bibitem{BouTou04}
B.~Bouchard and N.~Touzi.
\newblock Differentiability of quadratic {BSDE}s generated by continuous
  martingales.
\newblock {\em Stoch. Proc. Appl.}, 111:175--206, 2004.

\bibitem{BriandConfortola}
P.~Briand and F.~Confortola.
\newblock {BSDE}s with stochastic lipschitz condition and {Q}uadratic {PDE}s in
  {H}ilbert spaces.
\newblock 118:818--838, 2008.

\bibitem{BriandHu}
P.~Briand and Y.~Hu.
\newblock {BSDE}s with quadratic growth and unbounded terminal value.
\newblock {\em Probability Theory and Related Fields}, 136:604--619, 2006.

\bibitem{BriandHu08}
P.~Briand and Y.~Hu.
\newblock Quadratic {BSDE}s with convex generators and unbounded terminal
  conditions.
\newblock {\em Probability Theory and Related Fields}, 141:543--567, 2008.

\bibitem{ChasRich}
J.F. Chassagneux and A.~Richou.
\newblock Numerical simulation of quadratic {BSDE}s.
\newblock {\em The Annals of Applied Probability}, 26:262--304, 2016.

\bibitem{Chev97}
D.~Chevance.
\newblock Numerical methods for backward stochastic differential equations.
\newblock {\em Stoch. Proc. Appl.}, 111:175--206, 2004.

\bibitem{DelbaenHuRichou}
F.~Delbaen, Y.~Hu, and A.~Richou.
\newblock On the uniqueness of solutions to quadratic {BSDE}s with convex
  generators and unbounded terminal conditions.
\newblock {\em Ann. Inst. Henri Poincar\'{e} Probab. Stat.}, 150:145--192,
  2011.

\bibitem{DOP08}
G.~Di{N}unno, B.~{\O}ksendal, and F.~Proske.
\newblock {\em {M}alliavin {C}alculus for {L}\'evy {P}rocesses with
  {A}pplications to {F}inance}.
\newblock Springer, 2008.

\bibitem{FlanGubiPrio10}
F.~Flandoli, M.~Gubinelli, and E.~Priola.
\newblock Well posedness of the transport equation by stochastic pertubation.
\newblock {\em Invent math}, 180:1--53, 2009.

\bibitem{FreiDosReis}
C.~Frei and G.~Dos Reis.
\newblock {Q}uadratic {FBSDE}s with generalized {B}urger's type nonlinearities,
  pertubations and large deviations.
\newblock {\em Stoch. and Dynamics}, 13:1418--1453, 2013.

\bibitem{GeissYlinen}
S.~Geiss and J.~Ylinen.
\newblock {\em Decoupling on the {Wiener} {S}pace, {R}elated {B}esov {S}paces,
  and {A}pplications to {BSDE}s}, volume 272 of {\em Memoirs of the {A}merican
  {M}athematical {S}ociety}.
\newblock AMS, 2021.

\bibitem{HuImkMuller}
Y.~Hu, P.~Imkeller, and M.~Muller.
\newblock Utility maximization in incomplete markets.
\newblock {\em Annals of Applied Probability}, 15:1691--1712, 2005.

\bibitem{ImkDosReis}
P.~Imkeller and G.~Dos Reis.
\newblock Corrigendum to "{P}ath regularity and explicit convergence rate for
  {BSDE} with truncated quadratic growth" [stochastic process.appl. 120 (2010)
  348-379].
\newblock {\em Stochastic Process.Appl.}, 120:2286--2288, 2010.

\bibitem{ImkDos}
P.~Imkeller and G.~Dos Reis.
\newblock Path regularity and explicit convergence rate for bsde with truncated
  quadratic growth.
\newblock {\em Stoch. Process. Appl.}, 120:348--379, 2010.

\bibitem{ImkRevRich}
P.~Imkeller, A.~Reveillac, and A.~Richter.
\newblock Differentiability of quadratic {BSDE}s generated by continuous
  martingales.
\newblock {\em Annals of Applied Probability}, 22:907--941, 2012.

\bibitem{KPQ97}
N.~El Karoui, S.~Peng, and M.~C. Quenez.
\newblock Backward stochastic differential equations in finance.
\newblock {\em Math. Finance}, 7:1--77, 1997.

\bibitem{Kazamaki}
N.~Kazamaki.
\newblock {\em Continuous exponential martingales and {BMO}}, volume 1579 of
  {\em Lecture Notes in Mathematics}.
\newblock Springer-Verlag, 1994.

\bibitem{Kobylansky}
M.~Kobylanski.
\newblock Backward stochastic differential equations and partial differential
  equations with quadratic growth.
\newblock {\em Annals of Probability}, 2:558--602, 2000.

\bibitem{Kun90}
H.~Kunita.
\newblock {\em Stochastic {F}lows and {S}tochastic {D}ifferential {E}quations}.
\newblock Cambridge University Press, 1990.

\bibitem{Khoa}
Khoa L\^{e}.
\newblock Stochastic sewing lemma and applications.
\newblock {\em Electronic Journal of Probability}, 25(38):1--55, 2020.

\bibitem{MMNPZ13}
O.~Menoukeu-Pamen, T.~Meyer-Brandis, T.~Nilssen, F.~Proske, and T.~Zhang.
\newblock A variational approach to the construction and malliavin
  differentiability of strong solutions of {SDE}'s.
\newblock {\em Math. Ann.}, 357(2):761--799, 2013.

\bibitem{MenTan19}
O.~Menoukeu-Pamen and L.~Tangpi.
\newblock Strong solutions of some one-dimensional {SDE}s with random and
  unbounded drifts.
\newblock {\em SIAM J. Math. Anal.}, 51:4105--4141, 2019.

\bibitem{MiNuaSan89}
A.~Millet, D.~Nualart, and M.~Sanz.
\newblock Integration by parts and time reversal for diffusion processes.
\newblock {\em Annals of Probability}, 17:208--238, 1989.

\bibitem{NilssenProske}
S.~E.~A. Mohammed, T.~Nilssen, and F.~Proske.
\newblock Sobolev differentiable stochastic flows for sde's with singular
  coefficients: Applications to the stochastic transport equation.
\newblock {\em Annals of Probability}, 43(3):1535--1576, 2015.

\bibitem{Morlais09}
M.~Morlais.
\newblock Quadratic {BSDE}s driven by continuous martingale and application to
  the utility maximization problem.
\newblock {\em Finance Stoch.}, 13:121--150, 2009.

\bibitem{NAO21}
B.~Negyesi, K.~Andersson, and C.~Oosterlee.
\newblock The one step malliavin scheme: new discretization of {BSDE}s
  implemented with deep learning regressions.
\newblock arXiv:2110.05421v1, 03 2021.

\bibitem{Nua06}
D.~Nualart.
\newblock {\em The {M}alliavin Calculus and Related Topics}.
\newblock Springer Berlin, 2nd edition edition, 2006.

\bibitem{Pardoux-Peng92}
E.~Pardoux and S.~Peng.
\newblock Backward stochastic differential equations and quasilinear parabolic
  partial differential equations.
\newblock In B.L. Rozuvskii and R.B. Sowers, editors, {\em Stochastic partial
  differential equations and their applications}, volume 176, pages 200--217.
  Springer, Berlin, New York, 1992.

\bibitem{RhOlivOuk}
R.~Likibi Pellat, O.~Menoukeu Pamen, and Y.~Ouknine.
\newblock A class of quadratic forward-backward stochastic differential
  equations.
\newblock {\em Journal of Mathematical Analysis and Applications}, 2022.

\bibitem{DosReis}
G.~Dos Reis.
\newblock {\em On some properties of solutions to quadratic growth {BSDE} and
  applications to finance and insurance}.
\newblock PhD thesis, Humboldt University, 2010.

\bibitem{Richou}
A.~Richou.
\newblock Numerical simulation for bsdes with drivers of quadratic growth.
\newblock {\em The Annals of Applied Probability}, 21:1933--1964, 2011.

\bibitem{MPR17}
T.Mastrolia, D.Possamai, and A.~Reveillac.
\newblock On the {M}alliavin differentiability of {BSDE}s.
\newblock {\em Annal. IHP}, 53:2817--2857, 2017.

\bibitem{Zhang041}
J.~Zhang.
\newblock A numerical scheme for {BSDEs}.
\newblock {\em Ann. Appl. Probab.}, 14:459--488, 2004.

\end{thebibliography}

\end{document}